\documentclass{amsart}

\usepackage{color}
\usepackage{amssymb}
\usepackage{enumerate}

\newtheorem{thm}{Theorem}[section]
\newtheorem{prop}[thm]{Proposition}
\newtheorem{lem}[thm]{Lemma}
\newtheorem{cor}[thm]{Corollary}
\newtheorem{defi}[thm]{Definition}

\newtheorem{fact}[thm]{Fact}
\newtheorem{rem}[thm]{Remark}

\newcommand{\dbar}{{\mkern3mu\mathchar'26\mkern-12mu d}}

\numberwithin{equation}{section}

\begin{document}
	
\title[Asymptotics of singular values for quantum derivatives]{Asymptotics of singular values\\ for quantum derivatives}

\author{Rupert L. Frank}
\address[Rupert L. Frank]{Mathe\-matisches Institut, Ludwig-Maximilans Universit\"at M\"unchen, The\-resienstr.~39, 80333 M\"unchen, Germany, and Munich Center for Quantum Science and Technology, Schel\-ling\-str.~4, 80799 M\"unchen, Germany, and Mathematics 253-37, Caltech, Pasa\-de\-na, CA 91125, USA}
\email{r.frank@lmu.de}

\author{Fedor Sukochev}
\address[Fedor Sukochev]{School of Mathematics and Statistics, University of New South Wales, Kensington, 2052, NSW, Australia.}
\email{f.sukochev@unsw.edu.au}

\author{Dmitriy Zanin}
\address[Dmitriy Zanin]{School of Mathematics and Statistics, University of New South Wales, Kensington, 2052, NSW, Australia.}
\email{d.zanin@unsw.edu.au}

\begin{abstract} 
	We obtain Weyl type asymptotics for the quantised derivative $\dbar f$ of a function $f$ from the homgeneous Sobolev space $\dot{W}^1_d(\mathbb{R}^d)$ on $\mathbb{R}^d.$ The asymptotic coefficient $\|\nabla f\|_{L_d(\mathbb R^d)}$ is equivalent to the norm of $\dbar f$ in the principal ideal $\mathcal{L}_{d,\infty},$ thus, providing a non-asymptotic, uniform bound on the spectrum of $\dbar f.$ Our methods are based on the $C^{\ast}$-algebraic notion of the principal symbol mapping on $\mathbb{R}^d$, as developed recently by the last two authors and collaborators.
\end{abstract}

\maketitle

\section{Introduction and main results}

The topic of quantum derivatives lies at the intersection of various fields in mathematics, from noncommutative geometry over spectral and operator theory to real and harmonic analysis. Our specific goal in this paper is to derive spectral asymptotics for quantum derivatives, and we will do this under minimal regularity assumptions. These asymptotics will be deduced from spectral asymptotics for a certain class of pseudodifferential operators, which we also prove here and which are of independent interest. 

The idea of Connes' Quantum Calculus \cite{Connes-dga, Connes-book} is to replace topological spaces with $C^{\ast}$-algebras and Riemannian manifolds with spectral triples $(\mathcal{A},\mathcal{H},\mathcal{D}).$ Here, the $\ast$-algebra $\mathcal{A}$ represented on a Hilbert space $\mathcal{H}$ should be seen as a generalisation of $C^{\infty}(X),$ where $X$ is a Riemannian manifold. The operator $\mathcal{D}$ should be seen as a non-commutative version of a Dirac operator, which defines the differential calculus on the Riemannian manifold (e.g., differential forms are the linear span of $a_1[\mathcal{D},a_2]$ with $a_1,a_2\in\mathcal{A}$).

However, rough features of the manifold $X$ (e.g.\ its conformal geometry) do not need the full information about the operator $\mathcal{D}.$ Namely, Connes links conformal geometry with the sign of $\mathcal{D}$; see  \cite{Connes-dga, Connes-action}. In this context, a differential form $a_1[\mathcal{D},a_2],$ $a_1,a_2\in\mathcal{A},$ is replaced with its so-called quantised form $a_1[{\rm sgn}(\mathcal{D}),a_2].$ The operator $\dbar a=i[{\rm sgn}(\mathcal{D}),a]$  is called the quantised derivative of $a\in\mathcal{A}.$
	
	Of particular interest is the membership of $\dbar a$ to some ideal of compact operators. The compact operators on $\mathcal H$ are described by Connes as being analogous to infinitesimals, and the rate of decay of the sequence $\mu(T)=\{\mu(n,T)\}_{n=0}^\infty$ of singular values  corresponds in some way to the ``size" of the infinitesimal $T$; see \cite{Connes-book,Connes-LNM}. In this setting one can quantify the smoothness of an element $a \in \mathcal{A}$ in terms of the rate of decay of $\mu(\dbar a).$ Of particular interest are those elements $a \in \mathcal{A}$ that satisfy:
    \begin{align*}
                            \mu(n,\dbar a) &= O((n+1)^{-1/p}),\text{ or,}\\
        \sum_{n=0}^\infty \mu(n,\dbar a)^p &< \infty,\text{ or,}\\
        \sup_{n \geq 0} \frac{1}{\log(n+2)} \sum_{k=0}^n \mu(k,\dbar a)^p &< \infty\,,
    \end{align*}
    for some $p \in (0,\infty)$. The first condition stated above means that $\dbar a$ is in the weak-Schatten ideal $\mathcal{L}_{p,\infty}$, the second condition
    means that $\dbar a$ is in the Schatten ideal $\mathcal{L}_p$, and the final condition is that $|\dbar a|^p$ is in the Macaev--Dixmier ideal $\mathcal{M}_{1,\infty}$ \cite[Chapter 4, Section 2.$\beta$]{Connes-book}; see also \cite[Example 2.6.10]{LSZ-book}.

Concrete studies of quantised derivatives in various classical settings are available in the literature: for classical Riemannian manifolds we refer to \cite[Theorem 3]{Connes-action} and for noncommutative tori and for noncommutative Euclidean spaces to \cite{MSX-0,MSX}. In particular, for compact Riemannian manifolds, Connes established (see \cite[Theorem 3]{Connes-action}) that $|\dbar f|^{{\rm dim}(X)}$ belongs to Dixmier--Macaev ideal for every $f\in C^{\infty}(X).$

Let us now step back from the case of general manifolds and consider Euclidean space $\mathbb{R}^d$. We ask the following question: ``{\it For which functions $f$ on $\mathbb{R}^d,$ does their quantised derivative belong to a particular Schatten ideal?}". This question ties together several themes in operator theory and harmonic analysis. At the moment, the answer is known for the Schatten $\mathcal{L}_p$ and $\mathcal{L}_{p,\infty}$ ideals and is given in terms of some classical (Sobolev, Besov, etc) function spaces. Below, we recall a few key results.

\subsection{Spectral asymptotics for quantum derivatives}

Let us briefly set-up our notation. We work on $\mathbb R^d$ with $d\geq 2$ and denote variables by $x=(x_1,\ldots,x_d)$ and derivatives by $\nabla =(\partial_1,\ldots,\partial_d)$. We also write $D_j = -i\partial_j$. Let $N:=2^{[d/2]}$. There are Hermitian $N\times N$ matrices $\gamma_1,\ldots,\gamma_d$ such that
$$
\gamma_j \gamma_k + \gamma_k \gamma_j = 2 \delta_{j,k}
\qquad\text{for all}\ 1\leq j,k\leq d \,.
$$
Different choices of these matrices lead to equivalent results and we fix one such choice and define the Dirac operator
$$
\mathcal D = \sum_{j=1}^d \gamma_j \otimes D_j \,.
$$
This is an unbounded linear operator in the Hilbert space $\mathbb C^N\otimes L_2(\mathbb R^d)$. We are interested in its sign, defined by the functional calculus or, equivalently, by
$$
\mathrm{sgn}\, \mathcal D := \sum_{j=1}^d \gamma_j \otimes \frac{D_j}{\sqrt{D_1^2+\ldots + D_d^2}} \,.
$$
Given a measurable scalar function $f$ on $\mathbb{R}^d$ we denote by $M_f$ the linear operator in $L_2(\mathbb R^d)$ of pointwise multiplication by $f$. We are interested in the operator
$$
\dbar f := i \, [\mathrm{sgn}\,\mathcal D,1\otimes M_f]
$$
acting in $\mathbb C^N\otimes L_2(\mathbb R^d)$.

While for $f\in L_\infty(\mathbb R^d)$ the operator $\dbar f$ is clearly bounded in $\mathbb C^N\otimes L_2(\mathbb R^d)$, it also extends to a bounded operator for a certain class of unbounded functions $f$. Indeed, it is a result by Coifman, Rochberg and Weiss \cite{CoRoWe} that $\dbar f$ extends to a bounded operator in $\mathbb C^N\otimes L_2(\mathbb R^d)$ if and only if
$$
f\in BMO(\mathbb R^d) \,.
$$
The latter condition means, by definition, that $f\in L_{1,{\rm loc}}(\mathbb R^d)$ and
$$
\sup_{a\in\mathbb R^d, \, r>0} |B_r(a)|^{-1} \int_{B_r(a)} \left| f(x) - |B_r(a)|^{-1} \int_{B_r(a)} f(y)\,dy \right|dx <\infty \,.
$$
Here $B_r(a):= \{ x\in\mathbb R^d:\ |x-a|<r\}$. For a textbook presentation of this result we refer, for instance, to \cite[Section 3.5]{Grafakos}. Moreover, as shown in \cite{Uc}, $\dbar f$ is compact in $\mathbb C^N\otimes L_2(\mathbb R^d)$ if and only $f\in VMO(\mathbb R^d)$\footnote{In the literature, different spaces are denoted by VMO. We use the definition in \cite[Subsection 6.8]{Stein-book}. In \cite{Uc}, $VMO(\mathbb{R}^d)$ is denoted by $CMO(\mathbb{R}^d).$ In particular, our definition differs from the original one in \cite{Sa}.}, which is, by definition, the closure in $BMO(\mathbb{R}^d)$ of continuous, compactly supported functions.

In this paper, we are interested in quantitative compactness properties of $\dbar f$, measured in terms of the decay of its singular values. We will introduce our notation for singular values and the Schatten spaces in Subsection \ref{sec:compact} below. We recall that Janson and Wolff \cite{JaWo} have characterized membership of $\dbar f$ to the Schatten space $\mathcal L_p$ for $d<p<\infty$ in terms of membership of $f$ to a certain homogeneous fractional Sobolev space (whose definition we do not recall here, since it will play no role in what follows). Moreover, at the endpoint case $p=d$ they showed that $\dbar f\in \mathcal L_d$ if and only if $f$ is constant. The latter result was improved by Rochberg and Semmes \cite{RoSe2} on the Lorentz scale and then further in \cite{Fr}.

Of particular interest is the endpoint case of membership of $\dbar f$ to the weak trace ideal $\mathcal L_{d,\infty}$. To set the stage, we state the following theorem whose history we discuss momentarily. We denote by $\dot W_d^{1}(\mathbb R^d)$ the space of all $f\in L_{1,{\rm loc}}(\mathbb R^d)$ whose distributional gradient belongs to $L_d(\mathbb R^d,\mathbb{C}^d)$. Note also that, by the Poincar\'e inequality, $\dot W^{1}_d(\mathbb R^d)\subset BMO(\mathbb R^d)$. We refer the reader to the books \cite{Adams-book,Leoni-book,Taylor-book} for further information on Sobolev spaces.

\begin{thm}\label{mainiff}
Let $d\geq 2$ and $f\in BMO(\mathbb R^d)$. Then $\dbar f \in\mathcal L_{d,\infty}$ if and only if $f\in\dot W^1_d(\mathbb R^d)$. Moreover, with two constants $0<c_d\leq C_d<\infty$ depending only on $d$,
$$c_d \| \nabla f \|_{L_d(\mathbb{R}^d,\mathbb{C}^d)} \leq \| \dbar f\|_{\mathcal L_{d,\infty}} \leq C_d \|\nabla f\|_{L_d(\mathbb{R}^d,\mathbb{C}^d)}.$$
\end{thm}

To lighten the notations, in what follows we frequently write $\|\cdot \|_d$ instead of $\|\cdot\|_{L_d(\mathbb{ R}^d)}$ or $\|\cdot\|_{L_d(\mathbb{R}^d,\mathbb{C}^d)}.$ The particular norm is always clear from the context. A similar convention is applied to the operator norms; for instance, we write $\|\cdot\|_{d,\infty}$ instead of $\|\cdot\|_{\mathcal{L}_{d,\infty}}.$

Let us discuss precursors of Theorem \ref{mainiff}. Its statement appears explicitly in the appendix of the paper \cite{CoSuTe} by Connes, Sullivan and Teleman, written jointly with Semmes. The argument there relies on a deep analysis by Rochberg and Semmes \cite{RoSe2} of singular values of a certain class of operators. The novel ingredient in \cite{CoSuTe} is a certain derivative-free characterization of the Sobolev space $\dot W^1_d(\mathbb R^d)$, whose proof they sketched; see also \cite{RoSe} for an earlier discussion of these derivative-free conditions. In the recent paper \cite{Fr}, one of us gave an alternative proof of the characterization in the appendix of \cite{CoSuTe}. Meanwhile, in \cite{LMSZ} two of us proved Theorem \ref{mainiff} under the additional assumption $f\in L_\infty(\mathbb R^d)$. The proof there is independent of the work of Rochberg--Semmes and Connes--Sullivan--Teleman--Semmes and uses, instead, some tools from operator theory and pseudodifferential operator theory. Our proof of the implication $f\in\dot W^1_d(\mathbb R^d)\implies \dbar f\in\mathcal L_{d,\infty}$ in Theorem \ref{mainiff} uses the result of \cite{LMSZ} as an input. Our proof of the converse implication $\dbar f\in\mathcal L_{d,\infty} \implies f\in\dot W^{1}_d(\mathbb R^d)$ is to a large extent independent of \cite{LMSZ} and relies on the proof of spectral asymptotics, which we discuss next.

Indeed, our main result concerning $\dbar f$ is the following theorem about its spectral asymptotics.

\begin{thm}\label{main cor} Let $d\geq 2.$ If $f\in\dot{W}^{1}_d(\mathbb{R}^d)$ is real-valued, then
$$
\lim_{t\to\infty}t^{\frac1d}\mu(t,\dbar f)= \kappa_d \|\nabla f\|_{L_d(\mathbb R^d)} \,,
$$
where
$$
\kappa_d = (2\pi)^{-1} \left( N d^{-1} \int_{\mathbb S^{d-1}} (1-s_d^2)^{\frac d2} d\mathbf s \right)^\frac1d.
$$	
\end{thm}

This result substantially strengthens the main result in \cite{LMSZ} concerning singular traces of the operator $|\dbar f|^d.$ For a detailed study of singular traces, we refer the reader to the books \cite{LSZ-book,LSZ-book-2}. Their applications in quantised calculus are given in \cite{LMSZ-book}.


The following corollary is an almost immediate consequence of Theorems \ref{mainiff} and~\ref{main cor}.

\begin{cor}\label{corconst} Let $d\geq 2$. If $f\in BMO(\mathbb{R}^d)$ satisfies $\lim_{t\to\infty} t^\frac1d\mu(t,\dbar f)=0$, then $f$ is constant.
\end{cor}

This corollary strengthens the Janson--Wolff and Rochberg--Semmes results mentioned above. Whether the same conclusion also holds under the weaker assumption $\liminf_{t\to\infty} t^\frac1d\mu(t,\dbar f)=0$ is not known; see \cite{Fr} for a result in this direction.

Theorems \ref{mainiff} and \ref{main cor} give a complete answer to the question about spectral asymptotics for $\dbar f$. They show that, in order to have $\dbar f\in\mathcal L_{d,\infty}$, it is necessary that $f\in\dot W^1_d(\mathbb R^d)$ and, for real-valued $f$, as soon as this assumption is satisfied, one has power-like spectral asymptotics. In particular, we see that the asymptotic coefficient (namely, $\|\nabla f\|_{L_d(\mathbb R^d)}$) provides a non-asymptotic, uniform bound on the spectrum.

The problem of deriving spectral asymptotics under minimal regularity assumptions was emphasized by Birman and Solomyak, who also pointed out the important role played by non-asymptotic, uniform bounds. Their work and that of their collaborators is summarized, for instance in \cite{BiSo1,BiSoQuant}.

One motivation for the work of Birman and Solomyak came from the study of eigenvalues of Schr\"odinger-type operators. The spectral asymptotics in this case are typically referred to as Weyl-type asymptotics. Non-asymptotic, uniform bounds in this area are, for instance, Lieb--Thirring and Cwikel--Lieb--Rozenblum inequalities. Important results in this context were obtained, for instance, in \cite{Ro0,Ro1,Ro2,Ro3}; see also the textbook \cite[Chapter 4]{FrLaWe} and the references therein, as well as the recent paper \cite{Fr0}. Our results in Theorems \ref{mainiff} and \ref{main cor} can be viewed as the analogues for $\dbar f$ of these results for Schr\"odinger operators. The asymptotics in Theorem \ref{main cor} are also of semi-classical nature, but they are somewhat more subtle because of cancelations in the commutator defining $\dbar f$.

To the best of our knowledge, Theorems \ref{mainiff} and \ref{main cor} do not follow directly from the work of Birman and Solomyak and their school. Probably some parts of the proof of Theorem \ref{main cor} could be abbreviated by referring to \cite{BiSo0}. Given the limited availability of the latter paper, we decided to give a complete, rather self-contained proof of Theorem \ref{main cor}. Morever, we widen our perspective and, instead of just studying $\dbar f$, we consider a large class of pseudodifferential operators. The connection between this class of pseudodifferential operators and our original object of interest, $\dbar f$, is probably not immediately obvious and will be clarified later in the proof of Proposition \ref{lowerbound}.

There is another line of works to which our paper is connected. It concerns commutators with singular integral operators and their higher order analogues. From this huge literature, we only cite \cite{CoRoWe,Uc,Ja,JaWo,JaPe,ArFiPe,RoSe2} and refer to the references therein. The one-dimensional case is somewhat different and has been treated in \cite{Pe,CoRo,Ro}. These classical works concern almost exclusively non-asymptotic, uniform bounds and do not consider asymptotics. Exceptions are, for instance, \cite{EnRo,LMSZ} where asymptotics are proved in the weaker sense of the existence of a singular trace. We believe that the techniques that we develop in this paper might be useful in some of the above mentioned situations and provide asymptotics there as well.

\subsection{Spectral asymptotics for classical pseudodifferential operators}

The following $C^{\ast}$-algebra $\Pi$ is the closure (in the uniform norm) of the $\ast$-algebra of all compactly supported {\it classical} pseudodifferential operators of order $0.$ However, we use an elementary definition of $\Pi$, which does not involve pseudodifferential operators. The idea to consider this closure may be discerned in \cite[Proposition 5.2]{Atiyah-Singer}. For recent developments of this idea we refer to \cite{DAO1,DAO2}.

\begin{defi}\label{pis def} Let $\pi_1:L_{\infty}(\mathbb{R}^d)\to B(L_2(\mathbb{R}^d)),$ $\pi_2:L_{\infty}(\mathbb{S}^{d-1})\to B(L_2(\mathbb{R}^d))$ be defined by setting, for all $f\in L_\infty(\mathbb R^d)$ and $g\in L_\infty(\mathbb S^{d-1})$,
$$
\pi_1(f)=M_f
\qquad\text{and}\qquad
\pi_2(g)=g(\frac{-i\nabla}{\sqrt{-\Delta}}) \,.
$$
In other words, $\pi_2(g)$ acts on $L_2(\mathbb{R}^d)$ as a (homogeneous) Fourier multiplier
$$(\mathcal{F}(\pi_2(g)\xi))(s)=g(\frac{s}{|s|})(\mathcal{F}\xi)(s),\quad \xi\in L_2(\mathbb{R}^d),\quad s\in\mathbb{R}^d.$$
Let
$$
\mathcal{A}_1=\mathbb{C}+C_0(\mathbb{R}^d)
\qquad\text{and}\qquad
\mathcal{A}_2=C(\mathbb{S}^{d-1}) \,,
$$
and let $\Pi$ be the $C^{\ast}$-subalgebra in $B(L_2(\mathbb{R}^d))$ generated by the algebras $\pi_1(\mathcal{A}_1)$ and $\pi_2(\mathcal{A}_2).$
\end{defi}

In the definition of $\mathcal A_1$, by $C_0(\mathbb R^d)$ we denote the set of continuous functions tending to zero at infinity.

It should be pointed out that the $\ast$-representations $\pi_1$ and $\pi_2$ are continuous in the strong operator topology. This fact will play a substantial role in our study.

We say that $T\in B(L_2(\mathbb{R}^d))$ is \emph{compactly supported from the right} if there is a $\phi\in C^{\infty}_c(\mathbb{R}^d)$ such that $T=T\pi_1(\phi)$.

According to \cite{DAO1} (where a much stronger result is given in Theorem 1.2) or \cite{DAO2} (where a very general result is given in Theorem 3.3 and examplified on p.\ 284), there is a $\ast$-homomorphism
$${\rm sym}:\Pi\to\mathcal{A}_1\otimes_{{\rm min}}\mathcal{A}_2=C(\mathbb{S}^{d-1},\mathbb{C}+C_0(\mathbb{R}^d))$$
such that, for all $f\in \mathbb C + C_0(\mathbb R^d)$ and $g\in C(\mathbb S^{d-1})$,
$${\rm sym}(\pi_1(f))=f\otimes 1
\qquad\text{and}\qquad
{\rm sym}(\pi_2(g))=1\otimes g.$$
This $\ast$-homomorphism is called a \emph{principal symbol mapping}. It properly extends the notion of the principal symbol of the classical pseudodifferential operator.

If $T\in\Pi$ is compactly supported from the right, then ${\rm sym}(T)\in C_c(\mathbb{R}^d\times\mathbb{S}^{d-1}).$ Indeed, if $T=T\pi_1(\phi)$ with $\phi\in C^{\infty}_c(\mathbb{R}^d),$ then
$${\rm sym}(T)={\rm sym}(T)\cdot {\rm sym}(\pi_1(\phi))={\rm sym}(T)\cdot (\phi\otimes 1).$$
Thus, ${\rm sym}(T)$ is supported on ${\rm supp}(\phi)\times\mathbb{S}^{d-1}.$

The following is our main result concerning spectral asymptotics for pseudodifferential operators.

\begin{thm}\label{main thm} Let $d\geq 2.$ If $T\in\Pi$ is compactly supported from the right, then
$$\lim_{t\to\infty}t^{\frac1d}\mu(t,T(1-\Delta)^{-\frac12})= d^{-\frac1d} (2\pi)^{-1} \|{\rm sym}(T)\|_{L_d(\mathbb{R}^d\times\mathbb{S}^{d-1})}.$$
\end{thm}

Let us show that these asymptotics are of semiclassical nature. We consider the special case $T=\pi_1(f)\pi_2(g)$. On the one hand, we have $({\rm sym} (T))(\mathbf t,\mathbf s) = f(\mathbf t)g(\mathbf s)$. On the other hand, the function
$$P(\mathbf t,\mathbf p) = f(\mathbf t)g(\frac{\mathbf p}{|\mathbf p|})(1+|\mathbf p|^2)^{-\frac12},\quad (\mathbf t,\mathbf p)\in\mathbb R^d\times\mathbb R^d,$$
is the `semiclassical analogue' of the operator $T(1-\Delta)^{-1}$, and if we denote by $\mu(P)$ its decreasing rearrangement with respect to the measure $m\times (2\pi)^{-d} m$ on $\mathbb R^d\times\mathbb R^d$, then we easily find that
$$
\lim_{t\to\infty} t^\frac1d \mu(t, P) = d^{-\frac1d} (2\pi)^{-1} \| f\|_d \|g\|_d \,.
$$
Thus,
$$
\lim_{t\to\infty}t^{\frac1d}\mu(t,T(1-\Delta)^{-\frac12})= \lim_{t\to\infty} t^\frac1d \mu(t, P) \,,
$$
which, indeed, shows the semiclassical nature of Theorem \ref{main thm}.

The proof of Theorem \ref{main thm} is, in some sense, divided into three main steps. The first step concerns establishing a priori bounds, and the second one concerns the proof of asymptotics in a simple model situation. The third step is the proof of the full result by combining the first two steps. The a priori bounds in the first step are the topic of Section 4. Here we rely on \cite{LeSZ,MSX,LMSZ} for some intermediary results, but the main result, Theorem \ref{commutator theorem}, is new and is of independent interest. As in many problems of a semiclassical character, the spectral estimates of Cwikel, Kato--Seiler--Simon and Solomyak play an fundamental role in its proof. More specific to the problem at hand and of crucial importance both here and in \cite{MSX,LMSZ} are techniques from double operator integrals; see also \cite{CPSZ-JOT,CPSZ-AJM,CSZ-Fourier,DDSZ}. The above-mentioned second and third steps are the topic of Sections 5, 6 and 7. In their technical implementation we have found it convenient to not distinguish too strictly between the second and third step and to take these steps simultaneously. So, the a priori bounds will already come in when proving the asymptotics in a model problem on the torus with simpler symbols (see, for instance, Proposition \ref{main prop special torus}). The fact that in our situation a priori bounds play a more important role than in problems of a similar nature, considered, for instance, in the works of Birman and Solomyak mentioned above, comes ultimately from the fact that we are dealing with semiclassical asymptotics whose first term vanishes.

Having described the content of Sections 4-7, let us briefly describe that of the remaining sections. 
In Section 2, we set up our notation and recall the relevant definitions for spaces of functions and operators. 
In Section 3, we present, in an abstract setting, the arguments that we will use in the analysis of the model problem on the torus with simple symbols. Those might be of interest in various other problems as well. In the final Section 8 we deduce Theorems \ref{mainiff} and \ref{main cor} from Theorem \ref{main thm}. This will be relatively straightforward, given the earlier work in \cite{LMSZ}.


\subsection*{Acknowledgements}

Partial support through U.S.~National Science Foundation grant DMS-1954995 and through the German Research Foundation grant EXC-2111-390814868 is acknowledged.


\section{Preliminaries and notations}

\subsection{Function spaces}\label{function spaces subsection} The weak space $L_{p,\infty}(\mathbb{R}^d)$ is defined in the usual way:
$$L_{p,\infty}(\mathbb{R}^d)=\Big\{f:\ \sup_{t>0}t^{\frac1p}\mu(t,f)<\infty\Big\},$$
$$\|f\|_{p,\infty}=\sup_{t>0}t^{\frac1p}\mu(t,f).$$
Here, $\mu(\cdot,f)$ is the nonincreasing rearrangement of $|f|.$

We denote by $(L_{p,\infty})_0(\mathbb{R}^d)$ the closure of all compactly supported functions in $L_{p,\infty}(\mathbb{R}^d).$ In other words,
$$(L_{p,\infty})_0(\mathbb{R}^d)=\Big\{f\in L_{p,\infty}(\mathbb{R}^d):\ \lim_{t\to\infty}t^{\frac1p}\mu(t,f)=0\Big\}.$$

\subsection{Operator spaces}\label{sec:compact}

The following material is standard; for more details we refer the reader to \cite{Simon-book,BiSo, LSZ-book-2}. Let $H$ be a complex separable Hilbert space, and let $B(H)$ denote the set of all bounded operators on $H$, equipped with the uniform (operator) norm $\|\cdot\|_{\infty}$. Let $\mathcal{K}(H)$ denote the ideal of compact operators on $H.$ Given $T\in\mathcal{K}(H),$ the sequence of singular values
$\mu(T) = \{\mu(k,T)\}_{k=0}^\infty$ is defined as:
$$\mu(k,T) = \inf\{\|T-R\|_{\infty}:\quad\mathrm{rank}(R) \leq k\}.$$
Equivalently, $\mu(\cdot,T)$ is the sequence of eigenvalues of $|T|$ arranged in nonincreasing order with multiplicities.

It is convenient to define a singular value function $t\mapsto\mu(t,T),$ $t>0,$ by the same formula
$$\mu(t,T) = \inf\{\|T-R\|_{\infty}:\quad\mathrm{rank}(R) \leq t\}.$$
Note that
$$\mu(\cdot,T)=\sum_{k\geq0}\mu(k,T)\chi_{(k,k+1)}.$$
In what follows, we do not distinguish between the singular value sequence and the singular value function. We often abbreviate $\mu(\cdot,T)$ by $\mu(T)$.

If $(T_k)_{k\geq0}\subset B(H)$ is a bounded sequence, then $\bigoplus_{k\geq0}T_k$ is understood as an element in $B(H\oplus H\oplus\cdots).$ This is a convenient notation due to the following fact: if the sequence $(T_k)_{k\geq0}\subset B(H)$ consists of pairwise orthogonal operators (i.e. $T_kT_l=T_k^{\ast}T_l=0$ whenever $k\neq l$), then
$$\mu(\sum_{k\geq0}T_k)=\mu(\bigoplus_{k\geq0}T_k).$$

Let $p \in (0,\infty).$ The Schatten class $\mathcal{L}_p$ is the set of operators $T$ in $\mathcal{K}(\mathcal{H})$ such that $\mu(T)$ is $p$-summable, i.e.\ an element of the sequence space $\ell_p.$ If $p\geq 1,$ then the $\mathcal{L}_p$-norm is defined as:
$$\|T\|_p := \|\mu(T)\|_p = \left(\sum_{k=0}^\infty \mu(k,T)^p\right)^{\frac1p}.$$
With this norm $\mathcal{L}_p$ is a Banach space and an ideal of $B(H).$

The weak Schatten class $\mathcal{L}_{p,\infty}$ is the set of operators $T$ such that $\mu(T)$ is an element of the weak $L_p$-space $l_{p,\infty}$, with quasi-norm:
$$\|T\|_{p,\infty} = \sup_{t>0}t^{\frac1p}\mu(t,T) < \infty.$$
It follows from the inequality
$$\mu(t,T+S)\leq\mu(\frac{t}{2},T)+\mu(\frac{t}{2},S),\quad t>0,$$
that we have the following quasi-triangle inequality in $\mathcal{L}_{p,\infty}:$
$$\|T+S\|_{p,\infty}\leq 2^{\frac1p}(\|T\|_{p,\infty}+\|S\|_{p,\infty}).$$
The space $\mathcal{L}_{p,\infty}$ is an ideal of $B(H)$. We also have the following form of H\"older's inequality,
\begin{equation}\label{holder inequality}
\|TS\|_{r,\infty} \leq c_{p,q}\|T\|_{p,\infty}\|S\|_{q,\infty}
\end{equation}
where $\frac{1}{r}=\frac{1}{p}+\frac{1}{q}$, for some constant $c_{p,q}.$ For a detailed discussion of the H\"older inequality and the precise value of the optimal constant $c_{p,q}$ we refer to \cite{SZ-holder-pams}.

We also need the ideal
$$(\mathcal{L}_{p,\infty})_0=\left\{T\in\mathcal{L}_{p,\infty}:\ \lim_{t\to\infty}t^{\frac1p}\mu(t,T)=0 \right\}.$$
The ideal $(\mathcal{L}_{p,\infty})_0$ is the closure of the ideal of all finite rank operators in the norm $\|\cdot\|_{p,\infty}.$ 

Note that if $T\in\mathcal{L}_{p,\infty}$ and $A$ is compact, then $TA,AT\in(\mathcal{L}_{p,\infty})_0.$ Similarly, if $T\in\mathcal{L}_{p,\infty}$ and $S\in(\mathcal{L}_{q,\infty})_0,$ then $TS,ST\in(\mathcal{L}_{r,\infty})_0,$ $\frac1r=\frac1p+\frac1q.$

The following lemma is well known. We provide a proof for the sake of completeness.

\begin{lem}\label{realpart}
Let $0<p<\infty$ and $A\in\mathcal L_{p,\infty}$. Then $\Re A, \Im A\in\mathcal L_{p,\infty}$ and
$$\|\Re A \|_{p,\infty},\|\Im A\|_{p,\infty} \leq 2^{\frac1p}\|A\|_{p,\infty}.$$
Moreover,
$$\limsup_{t\to\infty} t^\frac 1p \mu(t,\Re A),\limsup_{t\to\infty} t^\frac1p \mu(t,\Im A) \leq 2^{\frac1p}\limsup_{t\to\infty} t^\frac 1p \mu(t,A).$$
In particular, if $A\in(\mathcal L_{p,\infty})_0$, then $\Re A, \Im A\in (\mathcal L_{p,\infty})_0.$
\end{lem}

\begin{proof} By the quasi-triangle inequality, we have
$$\|\Re A\|_{p,\infty}=\frac12\|A+A^{\ast}\|_{p,\infty}\leq 2^{\frac1p}\|A\|_{p,\infty},$$
$$\|\Im A\|_{p,\infty}=\frac12\|A-A^{\ast}\|_{p,\infty}\leq 2^{\frac1p}\|A\|_{p,\infty}.$$
Similarly,
\begin{align*}
\limsup_{t\to\infty} t^\frac 1p \mu(t,\Re A)&=\frac12\limsup_{t\to\infty} t^\frac 1p \mu(t,A+A^{\ast})\\
&\leq \frac12\limsup_{t\to\infty} t^\frac 1p\big(\mu(\frac{t}{2},A)+\mu(\frac{t}{2},A^{\ast})\big)=2^{\frac1p}\limsup_{t\to\infty} t^\frac 1p \mu(t,A).
\end{align*}
Similarly, one proves the inequality for $\Im A.$
\end{proof}


\section{Abstract limit theorems}

\subsection{Birman--Solomyak limit lemmas}

Throughout this subsection, we fix a parameter $0<p<\infty$. For proofs of the following two lemmas we refer to \cite[Section 11.6]{BiSo}.

\begin{lem}\label{bs sep lemma} Let $A\in\mathcal{L}_{p,\infty}$ and $B\in(\mathcal L_{p,\infty})_0$. Then
	$$\limsup_{t\to\infty}t^{\frac1p}\mu(t,A+B)=\limsup_{t\to\infty}t^{\frac1p}\mu(t,A)$$
	and
	$$\liminf_{t\to\infty}t^{\frac1p}\mu(t,A+B)=\liminf_{t\to\infty}t^{\frac1p}\mu(t,A).$$
	In particular,
	$$\lim_{t\to\infty}t^{\frac1p}\mu(t,A+B)=\lim_{t\to\infty}t^{\frac1p}\mu(t,A) \,,$$
	provided that the limit on the right hand side exists.
\end{lem}

\begin{lem}\label{bs limit lemma} Let $(A_n)_{n\geq0}\subset\mathcal{L}_{p,\infty}$ be such that
\begin{enumerate}
\item $A_n\to A$ in $\mathcal{L}_{p,\infty};$
\item for every $n\geq0,$ the limit
$$\lim_{t\to\infty}t^{\frac1p}\mu(t,A_n)=c_n \qquad\text{exists}.$$
\end{enumerate}
Then the following limits exist and are equal,
$$\lim_{t\to\infty}t^{\frac1p}\mu(t,A)=\lim_{n\to\infty}c_n.$$
\end{lem}

The following lemma is elementary to prove, but useful in applications. Recall that the operators $(A_k)_{k\geq0}\subset B(H)$ are called \emph{pairwise orthogonal} if $A_kA_l=A_k^{\ast}A_l=0$ whenever $k\neq l.$

\begin{lem}\label{bs pairwise orthogonal lemma} Let $(A_k)_{k=0}^n\subset\mathcal{L}_{p,\infty}$ be a sequence of pairwise orthogonal operators such that for all $0\leq k\leq n$, the limits
$$\lim_{t\to\infty}t^{\frac1p}\mu(t,A_k)=c_k
\qquad\text{exist}.$$
Then
$$\lim_{t\to\infty}t^{\frac1p}\mu(t,\sum_{k=0}^nA_k)= \left( \sum_{k=0}^nc_k^p \right)^{\frac1p}.$$
\end{lem}


\subsection{New limit lemmas}

\begin{lem}\label{first abstract limit lemma} Let $p>0.$ Let $A\in\mathcal{L}_{p,\infty}$ be such that the limit
$$\lim_{t\to\infty}t^{\frac1p}\mu(t,A)
\qquad\text{exists}.$$
Let $(p_l)_{0\leq l<k}$ be pairwise orthogonal projections in $H$ summing up to $1$ such that
\begin{enumerate}
\item $[A,p_l]\in(\mathcal{L}_{p,\infty})_0$ for $0\leq l<k;$
\item the operators $(p_lAp_l)_{0\leq l<k}$ are pairwise unitarily equivalent.
\end{enumerate}
Then the following limits exist and are equal,
$$\lim_{t\to\infty}t^{\frac1p}\mu(t,Ap_l)=\lim_{t\to\infty}t^{\frac1p}\mu(t,p_lAp_l)=k^{-\frac1p}\lim_{t\to\infty}t^{\frac1p}\mu(t,A),\quad 0\leq l<k.$$
\end{lem}

\begin{proof} Set $B=\sum_{l=0}^{k-1}p_lAp_l.$ It follows from the assumption that
	$$
	B-A= \sum_{l=0}^{k-1} p_l [A,p_l] \in(\mathcal{L}_{p,\infty})_0.
	$$
	By the assumption on $A$ and by Lemma \ref{bs sep lemma}, we have
$$\lim_{t\to\infty}t^{\frac1p}\mu(t,B)=\lim_{t\to\infty}t^{\frac1p}\mu(t,A).$$
Since the operators $(p_lAp_l)_{0\leq l<k}$ are pairwise orthogonal and pairwise unitarily equivalent, it follows that
$$\mu(t,B)=\mu(\frac{t}{k},p_lAp_l),\quad t>0,\quad 0\leq l<k.$$
Hence, we have
$$\lim_{t\to\infty}t^{\frac1p}\mu(\frac{t}{k},p_lAp_l)=\lim_{t\to\infty}t^{\frac1p}\mu(t,A),\quad 0\leq l<k.$$
Equivalently, we have
$$\lim_{t\to\infty}t^{\frac1p}\mu(t,p_lAp_l)=k^{-\frac1p}\lim_{t\to\infty}t^{\frac1p}\mu(t,A),\quad 0\leq l<k.$$
Since, by assumption, $Ap_l - p_lAp_l = [A,p_l]p_l \in(\mathcal{L}_{p,\infty})_0,$ it now follows from Lemma \ref{bs sep lemma} that the following limits exist and satisfy
$$\lim_{t\to\infty}t^{\frac1p}\mu(t,Ap_l)=\lim_{t\to\infty}t^{\frac1p}\mu(t,p_lAp_l),\quad 0\leq l<k.$$
This proves the lemma.
\end{proof}

\begin{lem}\label{second abstract limit lemma} Let $p>0.$ Let $A\in\mathcal{L}_{p,\infty}$ and let $(p_l)_{0\leq l<k}$ be pairwise orthogonal projections in $H$ summing up to $1$ such that
\begin{enumerate}
\item for every $0\leq l<k,$ the limit
$$\lim_{t\to\infty}t^{\frac1p}\mu(t,p_lA)=c_l
\qquad\text{exists};$$
\item $[A,p_l]\in(\mathcal{L}_{p,\infty})_0$ for every $0\leq l<k$.
\end{enumerate}
Then
$$\lim_{t\to\infty}t^{\frac1p}\mu(t,A)=(\sum_{0\leq l<k}c_l^p)^{\frac1p}.$$
\end{lem}

\begin{proof} Since $p_lAp_l-p_lA = p_l[A,p_l] \in(\mathcal{L}_{p,\infty})_0,$ it follows from the first assumption and Lemma \ref{bs sep lemma} that
$$\lim_{t\to\infty}t^{\frac1p}\mu(t,p_lAp_l)=c_l,\quad 0\leq l<k.$$
Set $B=\sum_{l=0}^{k-1}p_lAp_l.$ Since the operators $(p_lAp_l)_{0\leq l<k}$ are pairwise orthogonal, it follows, using Lemma \ref{bs pairwise orthogonal lemma}, that
$$\lim_{t\to\infty}t^{\frac1p}\mu(t,B)=(\sum_{l=0}^{k-1}c_l^p)^{\frac1p}.$$
Since, by assumption, $B-A=\sum_{l=0}^{k-1} p_l [A,p_l] \in(\mathcal{L}_{p,\infty})_0$, the assertion follows from Lemma \ref{bs sep lemma}.
\end{proof}

In the following lemma, we deal with subsets $I\subset[0,1)^n$. We call such a subset a cube if $I=[\mathbf a,\mathbf b) = [a_1,b_1)\times\cdots[a_n,b_n)$ for some $\mathbf a=(a_1,\ldots,a_n)$, $\mathbf b=(b_1,\ldots,b_n)$ in $[0,1)^n$ with $b_n-a_n=\ell$ for all $n=1,\ldots,N$. Two cubes are congruent if they have the same value of $\ell$.

\begin{lem}\label{third abstract limit lemma}
Let $p>0$ and $n\in\mathbb{N}.$ Let $A\in\mathcal{L}_{p,\infty}$ be such that the limit
$$\lim_{t\to\infty}t^{\frac1p}\mu(t,A)
\qquad\text{exists}.$$
If $\nu:[0,1)^n\to B(H)$ is a spectral measure such that
\begin{enumerate}
\item\label{tall aa} for any congruent cubes $I_1,I_2\subset[0,1)^n,$ there is a unitary operator $U$ such that $A=U^{-1} AU$ and $\nu(I_1)=U^{-1}\nu(I_2)U;$
\item\label{tall ab} for every Borel set $I\subset [0,1)^n,$ we have
$$[A,\nu(I)]\subset(\mathcal{L}_{p,\infty})_0;$$
\item\label{tall ac} there is a function $\psi$ with $\psi(+0)=0$ such that, for every Borel set $I\subset[0,1)^n,$ we have
$$\limsup_{t\to\infty}t^{\frac1p}\mu(t,A\nu(I))\leq\psi(m(I));$$
\end{enumerate}
then, for every Borel set $I\subset[0,1)^m,$ one has
$$\lim_{t\to\infty}t^{\frac1p}\mu(t,A\nu(I))=m(I)^{\frac1p}\lim_{t\to\infty}t^{\frac1p}\mu(t,A).$$
\end{lem}

\begin{proof} To lighten the notations, we assume
$$\lim_{t\to\infty}t^{\frac1p}\mu(t,A)=1.$$	

Suppose first $I=[\frac{{\bf k}_0}{k},\frac{{\bf k}_0+{\bf 1}}{k})$ for some $k\in\mathbb{N}$ and for some ${\bf k}_0\in \{0,\cdots,k-1\}^n.$ Set $p_{{\bf k}}=\nu([\frac{{\bf k}}{k},\frac{{\bf k}+{\bf 1}}{k})),$ ${\bf k}\in\{0,\cdots,k-1\}^n.$ By assumption \eqref{tall aa}, the operators $p_{{\bf k}}Ap_{{\bf k}},$ ${\bf k}\in\{0,\cdots,k-1\}^n,$ are pairwise unitarily equivalent. By assumption \eqref{tall ab}, $[A,p_{{\bf k}}]\in(\mathcal{L}_{p,\infty})_0.$ Hence, the assumptions in Lemma \ref{first abstract limit lemma} hold for the operator $A$ and for the projections $(p_{{\bf k}})_{\{0,\cdots,k-1\}^n}.$ By Lemma \ref{first abstract limit lemma}, the assertion follows for such $I.$
	
Suppose $I$ is a finite union of cubes as in the preceding paragraph. That is, $I=\cup_{0\leq l<L} I_l,$ where $(I_l)_{0\leq l<L}$ are pairwise disjoint cubes as in the preceding paragraph. Set $p_l=\nu(I_l),$ $0\leq l<L.$ By the preceding paragraph, the first assumption in Lemma \ref{second abstract limit lemma} holds. By assumption \eqref{tall ab}, the second assumption in Lemma \ref{second abstract limit lemma} holds. By Lemma \ref{second abstract limit lemma}, the assertion for such $I$ follows.

Let now $I$ be arbitrary. Fix $\epsilon>0$ and choose $J$ as in the preceding paragraph such that
$$m(J\triangle A)<\epsilon,\quad \psi(m(J\triangle I))<\epsilon^{1+\frac1p}.$$
Recall the inequality
$$\mu(t_1+t_2,T+S)\leq\mu(t_1,T)+\mu(t_2,S).$$
Since
$$\nu(I),\nu(J)\leq \nu(I\cup J)=\nu(J)+\nu(I\backslash J)=\nu(I)+\nu(J\backslash I),$$
it follows that
$$\mu(t(1+\epsilon),A\nu(I))\leq\mu(t(1+\epsilon),A\nu(I\cup J))\leq \mu(t,A\nu(J))+\mu(t\epsilon,A\nu(I\backslash J)),$$
$$\mu(t(1+\epsilon),A\nu(J))\leq\mu(t(1+\epsilon),A\nu(I\cup J))\leq \mu(t,A\nu(I))+\mu(t\epsilon,A\nu(J\backslash I)).$$
Thus,
\begin{align*}
(1+\epsilon)^{-\frac1p}\limsup_{t\to\infty}t^{\frac1p}\mu(t,A\nu(I))&=\limsup_{t\to\infty}t^{\frac1p}\mu(t(1+\epsilon),A\nu(I))\\
&\leq\limsup_{t\to\infty}\mu(t,A\nu(J))+\mu(t\epsilon,A\nu(I\backslash J))\\
&\leq\limsup_{t\to\infty}t^{\frac1p}\mu(t,A\nu(J))+\limsup_{t\to\infty}t^{\frac1p}\mu(t\epsilon,A\nu(I\backslash J))\\
&=\limsup_{t\to\infty}t^{\frac1p}\mu(t,A\nu(J)) \\
& \quad +\epsilon^{-\frac1p}\limsup_{t\to\infty}t^{\frac1p}\mu(t,A\nu(I\backslash J)).
\end{align*}
Recall that the assertion is already proved for $J.$ By assumption \eqref{tall ac}, we have
\begin{equation}\label{tall eq0}
(1+\epsilon)^{-\frac1p}\limsup_{t\to\infty}t^{\frac1p}\mu(t,A\nu(I))\leq \nu(J)^{\frac1p}+\epsilon^{-\frac1p}\psi(m(I\backslash J))\leq (\nu(I)+\epsilon)^{\frac1p}+\epsilon.
\end{equation}
Similarly,
\begin{align*}
(1+\epsilon)^{-\frac1p}(m(I)-\epsilon)^{\frac1p}&\leq (1+\epsilon)^{-\frac1p}m(J)^{\frac1p}\\
&=\liminf_{t\to\infty}t^{\frac1p}\mu(t(1+\epsilon),A\nu(J))\\
&\leq\liminf_{t\to\infty}t^{\frac1p}\mu(t,A\nu(I))+\limsup_{t\to\infty}t^{\frac1p}\mu(t\epsilon,A\nu(J\backslash I))\\
&\leq\liminf_{t\to\infty}t^{\frac1p}\mu(t,A\nu(I))+\epsilon^{-\frac1p}\limsup_{t\to\infty}t^{\frac1p}\mu(t,A\nu(J\backslash I)).
\end{align*}
By assumption \eqref{tall ac}, we have
\begin{align}\label{tall eq1}
(1+\epsilon)^{-\frac1p}(m(I)-\epsilon)^{\frac1p}
& \leq\liminf_{t\to\infty}t^{\frac1p}\mu(t,A\nu(I))+\epsilon^{-\frac1p}\psi(m(J\backslash I)) \\ & \leq\liminf_{t\to\infty}t^{\frac1p}\mu(t,A\nu(I))+\epsilon . \notag
\end{align}
Since $\epsilon>0$ in \eqref{tall eq0} and \eqref{tall eq1} is arbitrarily small, it follows that
$$\limsup_{t\to\infty}t^{\frac1p}\mu(t,A\nu(I))\leq\nu(I)^{\frac1p}\leq\liminf_{t\to\infty}t^{\frac1p}\mu(t,A\nu(I)).$$
This proves the assertion for an arbitrary $I.$
\end{proof}


\section{Commutator estimates}

In this section, we prove the following two results about commutators.

\begin{thm}\label{commutator theorem} Let $f\in L_{\infty}(\mathbb{R}^d)$ be compactly supported and let $g\in L_{\infty}(\mathbb{S}^{d-1}).$ We have
$$[\pi_1(f)(1-\Delta)^{-\frac12},\pi_2(g)]\in(\mathcal{L}_{d,\infty})_0,\quad [\pi_1(f),\pi_2(g)(1-\Delta)^{-\frac12}]\in(\mathcal{L}_{d,\infty})_0.$$
\end{thm}

Our second result is the analogue on the torus
$$
\mathbb T^d = (\mathbb R/2\pi\mathbb Z)^d \,.
$$
In the following, we slightly abuse the notation and denote
$$(-\Delta_{\mathbb{T}^d})^{-\frac12}\stackrel{def}{=}(-\Delta_{\mathbb{T}^d})^{-\frac12}\cdot (1-P),$$
where $P:L_2(\mathbb{T}^d)\to L_2(\mathbb{T}^d)$ is the orthogonal projection onto the subspace of constants. In particular, $-i\nabla_{\mathbb T^d}(-\Delta_{\mathbb T^d})^{-\frac12}$ is defined to vanish on constants.

\begin{thm}\label{toric commutator theorem}  Let $f\in L_{\infty}(\mathbb{T}^d)$ and let $g\in C(\mathbb{S}^{d-1}).$ We have
$$[M_f(1-\Delta)^{-\frac12},g\big(\frac{-i\nabla_{\mathbb{T}^d}}{\sqrt{-\Delta_{\mathbb{T}^d}}}\big)]\in(\mathcal{L}_{d,\infty})_0,\quad [M_f,g\big(\frac{-i\nabla_{\mathbb{T}^d}}{\sqrt{-\Delta_{\mathbb{T}^d}}}\big)(1-\Delta)^{-\frac12}]\in(\mathcal{L}_{d,\infty})_0.$$
\end{thm}

We only prove Theorem \ref{commutator theorem}. The proof of Theorem \ref{toric commutator theorem} is similar (but simpler in many aspects) and is, therefore, omitted.

Let us briefly discuss these results. Their main point is that they express a cancellation coming from the commutator. Indeed, the individual operators $\pi_1(f)(1-\Delta)^{-\frac12}\pi_2(g) = \pi_1(f)\pi_2(g)(1-\Delta)^{-\frac12}$ and $\pi_2(g)\pi_1(f)(1-\Delta)^{-\frac12}$, $\pi_2(g)(1-\Delta)^{-\frac12}\pi_1(f)$ belong\footnote{In the setting of torus, this is obvious. In the setting of Euclidean space the shortest (though not the easiest) way to see this is to use Theorem \ref{solomyak thm}. In Euclidean setting, it is of crucial importance that $f$ is compactly supported --- without this condition the operator $\pi_1(f)\pi_2(g)(1-\Delta)^{-\frac12}$ is not even compact (take $f\equiv 1$ and $g\equiv 1$).} to $\mathcal L_{d,\infty},$ but not to $(\mathcal L_{d,\infty})_0$ (unless $f\equiv 0$ or $g\equiv 0$).  According to Theorem \ref{commutator theorem}, when taking their difference, there is a significant cancellation and the resulting operator does belong to $(\mathcal L_{d,\infty})_0$.

We also remark that with slightly more effort one can show that the assertion remains valid under weaker conditions on $f$ and $g$. The above version, however, is sufficient for our purposes. What is important in our applications of this theorem is that no continuity of $f$ and $g$ is required (except for that of $g$ in Theorem \ref{toric commutator theorem}, which is needed to define the operator).


\subsection{Some trace ideal bounds}

\begin{lem}\label{d>2 useful cwikel} Let $d>2$. For $f\in L_\infty(\mathbb R^d)$ with compact support and $g\in L_\infty(\mathbb S^{d-1})$ we have
$$\|\pi_1(f)\pi_2(g)(-\Delta)^{-\frac12}\|_{d,\infty}\leq C_d\|f\|_d\|g\|_d.$$
\end{lem}

\begin{proof} Set
$$h(s)=g(\frac{s}{|s|})|s|^{-1},\quad s\in\mathbb{R}^d.$$
We have
$$\pi_1(f)\pi_2(g)(-\Delta)^{-\frac12}=M_fh(-i\nabla).$$
By Cwikel's estimate in $\mathcal{L}_{d,\infty},$ $d>2,$ (stated in \cite[Theorem 4.2]{Simon-book}; see also \cite{Cwikel} and a more general assertion in \cite{LeSZ}) we have
$$\|\pi_1(f)\pi_2(g)(-\Delta)^{-\frac12}\|_{d,\infty}\leq c_d\|f\|_d\|h\|_{d,\infty}.$$
Passing to spherical coordinates, we infer that
$$\|h\|_{d,\infty}=c_d'\|g\|_d.$$
The assertion follows by combining the two last equations.
\end{proof}

The following theorem is a restatement of the result of Solomyak (see \cite[Theorem 3.1]{Solomyak}) for $d=2.$

\begin{thm}\label{solomyak thm} If $f\in (L_2\log L)(\mathbb{R}^2)$ is supported on a compact set $K,$ then
	$$\|M_f(1-\Delta)^{-\frac12}\|_{2,\infty}\leq c(K)\|f\|_{L_2\log L}.$$
\end{thm}
Here, $L_2\log L$ is a common shorthand for the Orlicz space $L_M$ with $M(t)=t^2\log(e+t),$ $t>0.$

The next lemma is a substitute for Lemma \ref{d>2 useful cwikel} for the case $d=2.$

\begin{lem}\label{d=2 useful limsup} Let $d=2.$ For $f\in C^\infty_c(\mathbb{R}^2)$ and $g\in L_\infty(\mathbb{S}^1)$ we have
$$\limsup_{t\to\infty}t^{\frac12}\mu(t,\pi_1(f)\pi_2(g)(1-\Delta)^{-\frac12})\leq c_{{\rm abs}}\|f\|_2\|g\|_4.$$
\end{lem}

\begin{proof} There is a decomposition\footnote{Let $G({\bf t})=e^{-|{\bf t}|^2},$ ${\bf t}\in\mathbb{R}^2.$ Let $\phi\in C^{\infty}_c(\mathbb{R}^2)$ be such that $\|\phi\|_{\infty}=1$ and $\phi=1$ on ${\rm supp}(f).$ Setting $f_1=f(f^2+\|f\|_2^2 G^2)^{-\frac14}$ and $f_2=\phi(f^2+\|f\|_2^2 G^2)^{\frac14},$ we obtain the required decomposition.} $f=f_1f_2$ with $f_1,f_2\in C^{\infty}_c(\mathbb{R}^2)$ and such that $\|f_1\|_4\|f_2\|_4\leq c_{{\rm abs}}\|f\|_2.$ We write
\begin{align*}
&~\quad \pi_1(f)\pi_2(g)(1-\Delta)^{-\frac12}\\&=\pi_1(f_1)\cdot \pi_1(f_2)(1-\Delta)^{-\frac12}\cdot\pi_2(g)\\
&=\pi_1(f_1)(1-\Delta)^{-\frac14}\cdot \pi_1(f_2)(1-\Delta)^{-\frac14}\pi_2(g)\\
&\qquad 
+\pi_1(f_1)\cdot \Big(\pi_1(f_2)(1-\Delta)^{-\frac12}-(1-\Delta)^{-\frac14} \pi_1(f_2)(1-\Delta)^{-\frac14}\Big)\cdot\pi_2(g).
\end{align*}
Since the bracket in the last line belongs to $(\mathcal{L}_{2,\infty})_0$ (Theorem 1.6 in \cite{MSX} taken with $\alpha=-\frac12$ and $\beta=\frac12$ yields much stronger assertion that this bracket belongs to $\mathcal{L}_{1,\infty}$), it follows from Lemma \ref{bs sep lemma} and H\"older's inequality that
\begin{align*}
&~\quad \limsup_{t\to\infty}t^{\frac12}\mu\big(t,\pi_1(f)\pi_2(g)(1-\Delta)^{-\frac12}\big)\\
&=\limsup_{t\to\infty}t^{\frac12}\mu\big(t,\pi_1(f_1)(1-\Delta)^{-\frac14}\cdot \pi_1(f_2)(1-\Delta)^{-\frac14}\pi_2(g)\big)\\
&\leq\Big\|\pi_1(f_1)(1-\Delta)^{-\frac14}\cdot \pi_1(f_2)(1-\Delta)^{-\frac14}\pi_2(g)\Big\|_{2,\infty}\\
&\leq 2^{\frac12}\Big\|\pi_1(f_1)(1-\Delta)^{-\frac14}\Big\|_{4,\infty}\Big\|\pi_1(f_2)(1-\Delta)^{-\frac14}\pi_2(g)\Big\|_{4,\infty}.
\end{align*}
By Cwikel's estimate in $\mathcal{L}_{4,\infty}$ (stated in \cite[Theorem 4.2]{Simon-book}; see also \cite{Cwikel} and a more general assertion in \cite{LeSZ}), we have
$$\Big\|\pi_1(f_1)(1-\Delta)^{-\frac14}\Big\|_{4,\infty}\leq c_{{\rm abs}}\|f_1\|_4.$$
An argument identical to that in Lemma \ref{d>2 useful cwikel} yields
$$\Big\|\pi_1(f_2)(1-\Delta)^{-\frac14}\pi_2(g)\Big\|_{4,\infty}\leq c_{{\rm abs}}\|f_2\|_4\|g\|_4.$$
A combination of the three last estimates yields the claim.
\end{proof}


\subsection{A preliminary commutator estimate}

Our goal in this subsection is to prove the commutator estimate in Lemma \ref{pre commutator lemma}. We need some preparations.

\begin{lem}\label{first commutator lemma}
If $f\in C^{\infty}_c(\mathbb{R}^d),$ then
\begin{equation}\label{fcl1}
[(1-\Delta)^{\frac12},\pi_1(f)](1-\Delta)^{-1}\in(\mathcal{L}_{d,\infty})_0.
\end{equation}
\end{lem}

Note that, similarly to the discussion after Theorems \ref{commutator theorem} and \ref{toric commutator theorem}, the main point about this lemma is that it captures a cancellation. Indeed, both operators $\pi_1(f)(1-\Delta)^{-\frac12}$ and $(1-\Delta)^{\frac12}\pi_1(f)(1-\Delta)^{-1}$ belong to $\mathcal L_{d,\infty},$ but not to $(\mathcal L_{d,\infty})_0$ (unless $f=0$).

\begin{proof} The proof of this lemma requires some tools from the theory of Double Operator Integrals, for which we refer the reader to \cite{CPSZ-AJM,CSZ-Fourier,CPSZ-JOT,DDSZ}. The symbol $T^A_{\phi}$ denotes the double operator integral, based on the operator $A$, with the symbol $\phi$ as defined in these papers.
	
Let
$$\phi(\lambda,\mu)=\frac{\lambda^{\frac12}\mu^{\frac12}}{\lambda+\mu},\quad \lambda,\mu>0.$$
Note that
$$[(1-\Delta)^{\frac12},\pi_1(f)]=T^{(1-\Delta)^{\frac12}}_{\phi}\big((1-\Delta)^{-\frac14}[1-\Delta,\pi_1(f)](1-\Delta)^{-\frac14}\big).$$	
Thus,
$$[(1-\Delta)^{\frac12},\pi_1(f)](1-\Delta)^{-1}=T^{(1-\Delta)^{\frac12}}_{\phi}\big((1-\Delta)^{-\frac14}[1-\Delta,\pi_1(f)](1-\Delta)^{-\frac54}\big).$$	
By Lemma 6.2.3 in \cite{LMSZ-book}, the operator $T_{\phi}^{(1-\Delta)^{\frac12}}:(\mathcal{L}_{d,\infty})_0\to(\mathcal{L}_{d,\infty})_0$ is bounded. It, therefore, suffices to establish  	
$$(1-\Delta)^{-\frac14}[1-\Delta,\pi_1(f)](1-\Delta)^{-\frac54}\in\mathcal{L}_d\subset(\mathcal{L}_{d,\infty})_0.$$
	
To see the latter assertion, note that
\begin{align*}
	[1-\Delta,\pi_1(f)] & =\sum_{k=1}^d[D_k^2,\pi_1(f)]=\sum_{k=1}^d \left( D_k\pi_1(D_kf)+\pi_1(D_kf)D_k \right) \\
	& = - \pi_1(\Delta f)+2\sum_{k=1}^d\pi_1(D_kf)D_k.
\end{align*}
It, therefore, suffices to establish
$$(1-\Delta)^{-\frac14}\pi_1(D_kf)(1-\Delta)^{-\frac34}\in\mathcal{L}_d,\quad 1\leq k\leq d,\quad (1-\Delta)^{-\frac14}\pi_1(\Delta f)(1-\Delta)^{-\frac54}\in\mathcal{L}_d.$$
Furthermore, it suffices to show
$$\pi_1(D_kf)(1-\Delta)^{-\frac34}\in\mathcal{L}_d,\quad 1\leq k\leq d,\quad \pi_1(\Delta f)(1-\Delta)^{-\frac54}\in\mathcal{L}_d.$$
Both inclusions follow from the Kato--Seiler--Simon inequality (see, e.g., \cite[Theorem 4.1]{Simon-book}). This completes the proof.
\end{proof}


\begin{lem}\label{third commutator lemma}
If $f\in C^{\infty}_c(\mathbb{R}^d),$ then
$$\left[\pi_1(f),\frac{D_k}{(-\Delta)^{\frac12}}\right](1-\Delta)^{-\frac12}\in(\mathcal{L}_{d,\infty})_0,\quad 1\leq k\leq d.$$
\end{lem}
\begin{proof} Without loss of generality, assume that $f$ is real-valued. Let
$$g_k({\bf t}):=\frac{t_k}{|{\bf t}|}-\frac{t_k}{(1+|{\bf t}|^2)^{\frac12}}=\frac{t_k}{|{\bf t}|}\cdot\frac1{(1+|{\bf t}|^2)^{\frac12}(|{\bf t}|+(1+|{\bf t}|^2)^{\frac12})},\quad {\bf t}\in\mathbb{R}^d.$$
We decompose the operator $\left[\pi_1(f),\frac{D_k}{(-\Delta)^{1/2}}\right](1-\Delta)^{-\frac12}$ into four parts as follows:
$$[\pi_1(f),\frac{D_k}{(-\Delta)^{\frac12}}](1-\Delta)^{-\frac12}={\rm I}+{\rm II}+{\rm III}+{\rm IV},$$
where
$${\rm I}= \frac{D_k}{(1-\Delta)^{\frac12}}[(1-\Delta)^{\frac12},\pi_1(f)](1-\Delta)^{-1},\quad {\rm II}= [\pi_1(f),D_k](1-\Delta)^{-1},$$
$${\rm III}= \pi_1(f)g_k(-i\nabla)(1-\Delta)^{-\frac12},\quad {\rm IV}=-g_k(-i\nabla)\pi_1(f)(1-\Delta)^{-\frac12}.$$
	
It follows from Lemma \ref{first commutator lemma} that ${\rm I} \in(\mathcal{L}_{d,\infty})_0.$ Since $[D_k,\pi_1(f)]=\pi_1(D_k f),$ it follows from the Kato--Seiler--Simon inequality (see, e.g., \cite[Theorem 4.1]{Simon-book}) that ${\rm II},{\rm III},{\rm IV}\in\mathcal{L}_d\subset(\mathcal{L}_{d,\infty})_0.$ A combination of all four inclusions completes the proof.
\end{proof}


\begin{lem}\label{pre commutator lemma}
If $f\in C^{\infty}_c(\mathbb{R}^d)$ and if $g\in C^{\infty}(\mathbb{S}^{d-1}),$ then
$$[\pi_1(f),\pi_2(g)](1-\Delta)^{-\frac12}\in(\mathcal{L}_{d,\infty})_0.$$
\end{lem}
\begin{proof} Let $B=\frac{-i\nabla}{(-\Delta)^{\frac12}}=\left\{\frac{D_k}{(-\Delta)^{\frac12}}\right\}_{k=1}^d.$ We may extend\footnote{For example, we may set $h(t)=g(\frac{t}{|t|})\phi_0(|t|),$ $t\in\mathbb{R}^d,$ where $\phi_0$ is a Schwartz function on $\mathbb{R}$ which vanishes on $(-\infty,\frac12)$  and such that $\phi_0(1)=1.$} $g$ to a Schwartz function $h$ on $\mathbb{R}^d.$ The Fourier transform of $h$ is also a Schwartz function. By Definition \ref{pis def}, we have
$$\pi_2(g)=g(B)=h(B)=(2\pi)^{-\frac{d}{2}}\int_{\mathbb{R}^d}(\mathcal{F}h)(t)e^{i\langle B,t\rangle}dt.$$
Therefore,
\begin{equation}\label{precl eq0}
[\pi_1(f),\pi_2(g)](1-\Delta)^{-\frac12}=(2\pi)^{-\frac{d}{2}}\int_{\mathbb{R}^d}(\mathcal{F}h)(t)[\pi_1(f),e^{i\langle B,t\rangle}](1-\Delta)^{-\frac12}dt.
\end{equation}
	
An elementary computation yields
\begin{equation}\label{commutator with exponent int rep}
[x,e^{iy}]=i\int_0^1e^{isy}[y,x]e^{i(1-s)y}ds
\end{equation}
for all bounded operators $x$ and self-adjoint bounded operators $y.$ If $[x,y]\in(\mathcal{L}_{d,\infty})_0,$ then the intergand in \eqref{commutator with exponent int rep} belongs to $(\mathcal{L}_{d,\infty})_0$ and is weakly measurable. Since $(\mathcal{L}_{d,\infty})_0$ is a separable Banach space, the integrand is Bochner measurable and the integral can be understood in the Bochner sense in $(\mathcal{L}_{d,\infty})_0.$ Thus, $[x,e^{iy}]\in (\mathcal{L}_{d,\infty})_0$ and
\begin{equation}\label{commutator with exponent}
\|[x,e^{iy}]\|_{d,\infty}\leq\|[y,x]\|_{d,\infty}.
\end{equation}

Thus, for every $t\in\mathbb{R}^d,$
$$[\pi_1(f),e^{i\langle B,t\rangle}](1-\Delta)^{-\frac12}=[\pi_1(f)(1-\Delta)^{-\frac12},e^{i\langle B,t\rangle}]\in (\mathcal{L}_{d,\infty})_0$$
and
\begin{equation}\label{precl eq1}
\Big\|[\pi_1(f),e^{i\langle B,t\rangle}](1-\Delta)^{-\frac12}\Big\|_{d,\infty}\leq\sum_{k=1}^d|t_k|\cdot \Big\|[\pi_1(f),B_k](1-\Delta)^{-\frac12}\Big\|_{d,\infty},
\end{equation}
where the right hand side in \eqref{precl eq1} is finite by Lemma \ref{third commutator lemma}.

Hence, the integrand in \eqref{precl eq0} belongs to $(\mathcal{L}_{d,\infty})_0$ and is weakly measurable. Since $(\mathcal{L}_{d,\infty})_0$ is a separable Banach space, the integrand in \eqref{precl eq0} is Bochner measurable in $(\mathcal{L}_{d,\infty})_0$ and the integral in \eqref{precl eq0} can be understood in the Bochner sense in $(\mathcal{L}_{d,\infty})_0.$
\end{proof}


\subsection{Proof of Theorem \ref{commutator theorem}}

We are now ready to prove the main result of this section.

\begin{proof}[Proof of Theorem \ref{commutator theorem}]
	\emph{Step 1.} We prove the first assertion in the theorem under the additional assumption that $f\in C^\infty_c(\mathbb R^d)$.
	
	Let $h\in C^\infty(\mathbb S^{d-1})$ and write
	\begin{align*}
		[\pi_1(f)(1-\Delta)^{-\frac12},\pi_2(g)]
		& =[\pi_1(f)(1-\Delta)^{-\frac12},\pi_2(g-h)]+[\pi_1(f)(1-\Delta)^{-\frac12},\pi_2(h)] \\
		& =\pi_1(f)\pi_2(g-h)(1-\Delta)^{-\frac12}-(1-\Delta)^{-\frac12}\pi_2(g-h)\pi_1(f) \\
		& \quad -\pi_2(g-h)\cdot [\pi_1(f),(1-\Delta)^{-\frac12}] \\
		& \quad +[\pi_1(f),\pi_2(h)](1-\Delta)^{-\frac12}.
	\end{align*}
	Noting that $g-h\in L_\infty(\mathbb S^{d-1})$, the third summand belongs to $(\mathcal{L}_{d,\infty})_0$. Indeed, Theorem 1.6 in \cite{MSX} taken with $\alpha=-1$ and $\beta=0$ yields the stronger assertion that this summand belongs to $\mathcal{L}_{\frac{d}{2},\infty}$. The fourth summand belongs to $(\mathcal{L}_{d,\infty})_0$ by Lemma \ref{pre commutator lemma}. Therefore, it follows from Lemma \ref{bs sep lemma} that
	\begin{align*}
	&~ \quad \limsup_{t\to\infty}t^{\frac1d}\mu\Big(t,[\pi_1(f)(1-\Delta)^{-\frac12},\pi_2(g)]\Big) \\
	& = \limsup_{t\to\infty}t^{\frac1d}\mu\Big(t,\pi_1(f)\pi_2(g-h)(1-\Delta)^{-\frac12}-(1-\Delta)^{-\frac12}\pi_2(g-h)\pi_1(f)\Big) \\
	& \leq c_d \limsup_{t\to\infty}t^{\frac1d}\mu\Big(t,\pi_1(f)\pi_2(g-h)(1-\Delta)^{-\frac12}\Big) \\
	& \quad + c_d  \limsup_{t\to\infty}t^{\frac1d}\mu\Big(t,(1-\Delta)^{-\frac12}\pi_2(g-h)\pi_1(f)\Big) \\
	& \leq c_d' \|f\|_d \|g-h\|_p \,,
	\end{align*}
	where $p=d$ for $d>2$ and $p=4$ for $d=2$. The last inequality uses Lemmas \ref{d>2 useful cwikel} and \ref{d=2 useful limsup} for $d>2$ and $d=2$, respectively. Choosing a sequence of $h$'s that converges to $g$ in $L_p(\mathbb S^{d-1})$, we conclude that
	$$\limsup_{t\to\infty}t^{\frac1d}\mu\Big(t,[\pi_1(f)(1-\Delta)^{-\frac12},\pi_2(g)]\Big) = 0 \,,$$
	which is the claimed assertion.
	
	\medskip
	
\emph{Step 2.} We now prove the first assertion in the theorem for general $f\in L_\infty(\mathbb R^d)$ with compact support.
	
By applying dilation and translation, we may assume without loss of generality that $f$ is supported on $[0,1]^d.$
	
We choose a sequence $(f_n)_{n\geq1}\subset C^{\infty}_c(\mathbb{R}^d)$ supported in $[0,1]^d$ such that $\|f-f_n\|_d<\frac1n$ if $d>2$ and $\|f-f_n\|_{L_2\log L}<\frac1n$ if $d=2$. By the first step, we have
$$[\pi_1(f_n)(1-\Delta)^{-\frac12},\pi_2(g)]\in(\mathcal{L}_{d,\infty})_0.$$

On the other hand, by Cwikel's estimate in $\mathcal{L}_{d,\infty},$ $d>2,$ (stated as Theorem 4.2 in \cite{Simon-book}; see also \cite{Cwikel} and a more general assertion in \cite{LeSZ}) and Theorem \ref{solomyak thm} for $d=2$, we have
$$\|\pi_1(f_n-f)(1-\Delta)^{-\frac12}\|_{d,\infty}\leq\frac{c_d}{n},\quad n\geq 1.$$
This implies
\begin{align*}
&~\quad \Big\|[\pi_1(f_n)(1-\Delta)^{-\frac12},\pi_2(g)]-[\pi_1(f)(1-\Delta)^{-\frac12},\pi_2(g)]\Big\|_{d,\infty}\\
&\leq c_d'\|\pi_1(f_n-f)(1-\Delta)^{-\frac12}\|_{d,\infty}\|g\|_{\infty}\\
&\leq \frac{c_dc_d' \|g\|_\infty}{n},\quad n\geq 1.
\end{align*}
Thus,
$$[\pi_1(f_n)(1-\Delta)^{-\frac12},\pi_2(g)]\to [\pi_1(f)(1-\Delta)^{-\frac12},\pi_2(g)],\quad n\to\infty,$$
in $\mathcal{L}_{d,\infty}.$ Since $(\mathcal{L}_{d,\infty})_0$ is closed in $\mathcal L_{d,\infty}$ by Lemma \ref{bs limit lemma}, the first assertion follows.	

\medskip

\emph{Step 3.} We finally prove the second assertion in the theorem.

Indeed, it follows from the first one and the identity
$$[\pi_1(f),\pi_2(g)(1-\Delta)^{-\frac12}]=[\pi_1(f)(1-\Delta)^{-\frac12},\pi_2(g)]+\pi_2(g)\cdot [\pi_1(f),(1-\Delta)^{-\frac12}].$$
This completes the proof of Theorem \ref{commutator theorem}.
\end{proof}


\section{Proof of Theorem \ref{main thm} in a special case. Case $d>2$}

Our goal in this and the next section is to prove the following theorem, which is a special case of Theorem \ref{main thm} and contains the main difficulty.

\begin{thm}\label{main thm special}
Let $\mathcal{A}_1$ and $\mathcal{A}_2$ be as in Definition \ref{pis def}. Let $(f_n)_{n=1}^N\subset\mathcal{A}_1$ be compactly supported and let $(g_n)_{n=1}^N\subset\mathcal{A}_2.$ If $$T=\sum_{n=1}^N\pi_1(f_n)\pi_2(g_n),$$
then
$$
\lim_{t\to\infty}t^{\frac1d}\mu(t,T(1-\Delta)^{-\frac12})= d^{-\frac1d} (2\pi)^{-1} \|{\rm sym}(T)\|_d.
$$
\end{thm}

The proof of this theorem is somewhat different in dimensions $d\geq 3$ and $d=2$. In this section we deal with the former case and in the next one with the latter.


\subsection{Asymptotics on the torus}

Our goal in this subsection is to prove the following analogue of Theorem \ref{main thm special} for $d\geq 3$ on the torus $\mathbb T^d = (\mathbb R/2\pi\mathbb Z)^d$. We recall from the discussion before Theorem \ref{toric commutator theorem} our convention for the operator $(-\Delta_{\mathbb{T}^d})^{-\frac12}$.

\begin{thm}\label{main thm special torus}
	Let $d\geq 3$. Let $(f_n)_{n=1}^N\subset C(\mathbb T^d)$ and let $(g_n)_{n=1}^N\subset C(\mathbb S^{d-1})$. Then
	$$
	\lim_{t\to\infty}t^{\frac1d}\mu\Big(t, \sum_{n=1}^NM_{f_n}g_n\big(\frac{-i\nabla_{\mathbb{T}^d}}{\sqrt{-\Delta_{\mathbb{T}^d}}}\big)(-\Delta_{\mathbb{T}^d})^{-\frac12}\Big)
	= d^{-\frac1d} (2\pi)^{-1} \left\| \sum_{n=1}^N f_n\otimes g_n \right\|_d .
	$$
\end{thm}

We begin by proving the assertion for $N=1$ and $f_1=1.$

\begin{lem}\label{reduction to tensor product lemma} Let $g\in C(\mathbb{S}^{d-1}).$ We have
$$\lim_{t\to\infty}t^{\frac1d}\mu\Big(t,g\big(\frac{-i\nabla_{\mathbb{T}^d}}{\sqrt{-\Delta_{\mathbb{T}^d}}}\big)(-\Delta_{\mathbb{T}^d})^{-\frac12}\Big)=d^{-\frac1d}\, \|g\|_d.$$
\end{lem}

\begin{proof}
	\emph{Step 1.} We define the function $G_1$ on $\mathbb{R}^d$ by setting
$$G_1({\bf s})=g(\frac{{\bf s}}{|{\bf s}|})|{\bf s}|^{-1},\quad {\bf s}\in\mathbb{R}^d,$$
and claim that
$$\mu(t,G_1)= d^{-\frac1d} \|g\|_dt^{-\frac1d},\quad t>0.$$
Here, $\mu(\cdot,G_1)$ is a decreasing rearrangement of $|G_1|$ (as in the Subsection \ref{function spaces subsection}).

Consider the measure space $X=\mathbb{R}_+\times\mathbb{S}^{d-1}$ equipped with the product measure $\frac1d m_{\mathbb{R}_+}\times m_{\mathbb{S}^{d-1}},$ where $m_{\mathbb{R}_+}$ is the Lebesgue measure on $\mathbb{R}_+$ and $m_{\mathbb{S}^{d-1}}$ is the Lebesgue measure on the sphere. The spherical coordinate change
$${\bf s}\to (|{\bf s}|^d,\frac{{\bf s}}{|{\bf s}|}),\quad s\in\mathbb{R}^d,$$
preserves the measure (i.e., transforms Lebesgue measure on $\mathbb{R}^d$ to the above pro\-duct measure). Hence, this spherical coordinate change preserves the decreasing rearrangement. The image of $G_1$ under the spherical coordinate change is $z^{\frac1d}\otimes g,$ where $z(t)=t^{-1},$ $t>0.$ Thus,
$$\mu_{L_{\infty}(\mathbb{R}^d)}(G_1)=\mu_{L_{\infty}(\mathbb{R}_+\times\mathbb{S}^{d-1},\frac1d m_{\mathbb{R}_+}\times m_{\mathbb{S}^{d-1}})}(z^{\frac1d}\otimes g)=d^{-\frac1d} \|g\|_dz^{\frac1d},$$
where we use the notation $\mu_{L_{\infty}(X,\nu)}(x)$ to emphasise that $x$ is a $\nu$-measurable function on a measure space $(X,\nu)$ and that the decreasing rearrangement of $|x|$ is taken with respect to the measure $\nu.$ This immediately yields the assertion of Step~1.

\medskip
	
{\it Step 2.} We define the function $G_2$ on $\mathbb{R}^d$ by setting
$$G_2({\bf s})=g(\frac{{\lfloor {\bf s}\rfloor}}{|{\lfloor {\bf s}\rfloor}|})|{\lfloor {\bf s}\rfloor}|^{-1}\chi_{\mathbb{Z}^d\backslash\{0\}}(\lfloor {\bf s}\rfloor),\quad {\bf s}\in\mathbb{R}^d.$$
Here, $\lfloor {\bf s}\rfloor$ is a shorthand for the vector $(\lfloor s_1\rfloor,\cdots,\lfloor s_d,\rfloor).$ We claim that	
$$\lim_{t\to\infty}t^{\frac1d}\mu(t,G_2)=d^{-\frac1d} \|g\|_d.$$

Indeed, it follows from the continuity of $g$ that
$$G_1-G_2\in (L_{d,\infty})_0(\mathbb{R}^d),$$
so, by a simple commutative analogue of Lemma \ref{bs sep lemma}, the assertion of Step 2 follows from that of Step 1.

\medskip

{\it Step 3.} For the sequence (indexed by $\mathbb{Z}^d$)
$$G_3({\bf n})=\Big\{g(\frac{{\bf n}}{|{\bf n}|})|{\bf n}|^{-1}\chi_{\mathbb{Z}^d\backslash\{0\}}({\bf n})\Big\}_{{\bf n}\in\mathbb{Z}^d} $$
it is immediate that
$$\mu_{L_{\infty}(\mathbb{R}^d)}(G_2)=\mu_{l_{\infty}(\mathbb Z^d)}(G_3).$$
It follows from Step 2 that
$$\lim_{t\to\infty}t^{\frac1d}\mu_{l_{\infty}(\mathbb{Z}^d)}(t,G_3)= d^{-\frac1d} \|g\|_d.$$
Since the operator
$$g\big(\frac{-i\nabla_{\mathbb{T}^d}}{\sqrt{-\Delta_{\mathbb{T}^d}}}\big)(-\Delta_{\mathbb{T}^d})^{-\frac12}$$
is diagonal in the Fourier basis and since the corresponding sequence of diagonal entries is exactly $G_3,$ the assertion follows.
\end{proof}

Next, we prove Theorem \ref{main thm special torus} in the case $N=1$ and with $f_1$ being an indicator function.

\begin{prop}\label{main prop special torus}
	Let $d\geq 3$. Let $I \subset \mathbb T^d$ be a Borel subset and let $g\in C(\mathbb S^{d-1})$. Then
	$$
	\lim_{t\to\infty}t^{\frac1d}\mu\Big(t, M_{\chi_I}g\big(\frac{-i\nabla_{\mathbb{T}^d}}{\sqrt{-\Delta_{\mathbb{T}^d}}}\big)(-\Delta_{\mathbb{T}^d})^{-\frac12}\Big)
	= d^{-\frac1d} (2\pi)^{-1} m(I)^{\frac1d} \|g\|_d \,.
	$$
\end{prop}

\begin{proof} Throughout this proof, $g\in C(\mathbb S^{d-1})$ is fixed.	

We will apply Lemma \ref{third abstract limit lemma} with $p=d$ and
$$A=g\big(\frac{-i\nabla_{\mathbb{T}^d}}{\sqrt{-\Delta_{\mathbb{T}^d}}}\big)(-\Delta_{\mathbb{T}^d})^{-\frac12}.$$
Then, by Lemma \ref{reduction to tensor product lemma}, $A\in\mathcal L_{d,\infty}$ and
$$\lim_{t\to\infty} t^\frac1d \mu(t,A)=d^{-\frac1d} \|g\|_d.$$
	
	Let us define a spectral measure $\nu$ on $[0,1)^d$. For a Borel set $I\subset[0,1)^d$, we define a Borel set $\tilde I\subset\mathbb T^d$ by dilating it by a factor of $2\pi$ (resulting in a subset of $[0,2\pi)^d$), then extending it periodically to a subset of $\mathbb R^d$ and finally interpreting it as a subset of $\mathbb T^d$. If $\chi_{\tilde I}$ is the indicator function of $\tilde I$, then we set
	$$
	\nu(I) = M_{\chi_{\tilde I}} \,.
	$$
	This, indeed, defines a spectral measure.
	
	If $I_1,I_2\subset[0,1)^d$ are congruent cubes, then there is a translation $U\in B(L_2(\mathbb{T}^d))$ such that $\nu(I_1)=U^{-1}\nu(I_2)U.$ Clearly, translations commute with functions of the gradient and, therefore, with $A$. This verifies the first assumption in Lemma \ref{third abstract limit lemma}.
	
The second assumption in Lemma \ref{third abstract limit lemma} follows from  Theorem \ref{toric commutator theorem}. The third assumption in Lemma \ref{third abstract limit lemma} follows from the estimate
\begin{align*}
&~\quad \limsup_{t\to\infty}t^{\frac1d}\mu\Big(t, M_{\chi_{\tilde{I}}}g\big(\frac{-i\nabla_{\mathbb{T}^d}}{\sqrt{-\Delta_{\mathbb{T}^d}}}\big)(-\Delta_{\mathbb{T}^d})^{-\frac12}\Big)\\
&\leq\|g\|_{\infty}\limsup_{t\to\infty}t^{\frac1d}\mu\Big(t, M_{\chi_{\tilde{I}}}(-\Delta_{\mathbb{T}^d})^{-\frac12}\Big)\\
&\leq\|g\|_{\infty}\|M_{\chi_{\tilde{I}}}(-\Delta_{\mathbb{T}^d})^{-\frac12}\|_{d,\infty}\\
&\leq C_d\|g\|_{\infty}m(\tilde{I})^{\frac1d}.
\end{align*}
Here, the last step follows from the abstract Cwikel-type estimate (see \cite[Theorem 3.4 or Corollary 3.6]{LeSZ}).

Hence, one can apply Lemma \ref{third abstract limit lemma} and one obtains
	$$
	\lim_{t\to\infty} t^\frac1d \mu(t, A M_{\chi_{\tilde I}}) = d^{-\frac1d} 2\pi m(I)^\frac1d \|g\|_d = d^{-\frac1d} m(\tilde I)^\frac1d \|g\|_d \,.
	$$
	Because of the second assumption in Lemma \ref{third abstract limit lemma} and Lemma \ref{bs sep lemma}, this remains valid with $A M_{\chi_{\tilde I}}$ replaced by $M_{\chi_{\tilde I}}A$, which is the assertion of the proposition.
\end{proof}

We can now turn to the proof of the main result of this subsection.

\begin{proof}[Proof of Theorem \ref{main thm special torus}]
	For every $K\in\mathbb N$ and $\mathbf k\in\{0,\ldots,K-1\}^d$, let
	$$
	h_{\mathbf k,K}(\mathbf t) = \chi_{[\frac{2\pi\mathbf k}{K}, \frac{2\pi(\mathbf k+\mathbf 1)}{K})}(\mathbf t) \,,
	\qquad \mathbf t \in[0,2\pi)^d \equiv \mathbb T^d \,.
	$$
	Let $f_{n,K}$ be the conditional expectation on $L_\infty(\mathbb T^d)$ of $f_n$ onto the subalgebra generated by $(h_{\mathbf k,K})_{\mathbf k\in\{0,\ldots,K-1\}^d}$ and write
	$$
	f_{n,K} = \sum_{\mathbf k\in \{0,\ldots,K-1\}^d} c_{n,\mathbf k,K} h_{\mathbf k,K} \,.
	$$
	We set
	$$
	g_{\mathbf k,K} = \sum_{n=1}^N c_{n,\mathbf k,K} g_n
	$$
	and
	$$
	A_K = \sum_{\mathbf k\in\{0,\ldots,K-1\}^d} \pi_1(h_{\mathbf k,K}) \pi_2(g_{\mathbf k,K})(-\Delta_{\mathbb T^d})^{-\frac12} \pi_1(h_{\mathbf k,K}) \,.
	$$
	Since the operators appearing in the sum defining $A_K$ are pairwise orthogonal, we have
	$$
	\mu(A_K) = \mu\left( \bigoplus_{\mathbf k\in\{0,\ldots,K-1\}^d} \pi_1(h_{\mathbf k,K}) \pi_2(g_{\mathbf k,K})(-\Delta_{\mathbb T^d})^{-\frac12} \pi_1(h_{\mathbf k,K}) \right).
	$$
	By Proposition \ref{main prop special torus}, we have
	$$
	\lim_{t\to\infty} t^\frac1d \mu(t, \pi_1(h_{\mathbf k,K}) \pi_2(g_{\mathbf k,K})(-\Delta_{\mathbb T^d})^{-\frac12}) = d^{-\frac1d} K^{-1} \|g_{\mathbf k,K}\|_d \,.
	$$
	On the other hand, by Theorem \ref{toric commutator theorem},
	\begin{align*}
		&~\quad  \pi_1(h_{\mathbf k,K}) \pi_2(g_{\mathbf k,K})(-\Delta_{\mathbb T^d})^{-\frac12} \pi_1(h_{\mathbf k,K}) - \pi_1(h_{\mathbf k,K}) \pi_2(g_{\mathbf k,K})(-\Delta_{\mathbb T^d})^{-\frac12} \\
		&  = \pi_1(h_{\mathbf k,K}) [\pi_2(g_{\mathbf k,K})(-\Delta_{\mathbb T^d})^{-\frac12}, \pi_1(h_{\mathbf k,K}) ] \in (\mathcal L_{d,\infty})_0 \,,
	\end{align*}
	so, by Lemma \ref{bs sep lemma},
	$$
	\lim_{t\to\infty} t^\frac1d \mu(t, \pi_1(h_{\mathbf k,K}) \pi_2(g_{\mathbf k,K})(-\Delta_{\mathbb T^d})^{-\frac12}\pi_1(h_{\mathbf k,K})) = d^{-\frac1d} K^{-1} \|g_{\mathbf k,K}\|_d \,.
	$$
	We deduce, by Lemma \ref{bs pairwise orthogonal lemma},
	\begin{align}\label{eq:torusproof}
	&~\quad  \lim_{t\to\infty} t^\frac1d \mu\left( t,\bigoplus_{\mathbf k\in\{0,\ldots,K-1\}^d} \pi_1(h_{\mathbf k,K}) \pi_2(g_{\mathbf k,K})(-\Delta_{\mathbb T^d})^{-\frac12} \pi_1(h_{\mathbf k,K}) \right) \notag \\
	& = d^{-\frac1d} \left( \frac{1}{K^d} \sum_{\mathbf k\in\{0,\ldots,K-1\}^d} \|g_{\mathbf k,K}\|_d^d \right)^\frac1d.
	\end{align}

	Now let us introduce
	$$
	T_K = \sum_{n=1}^N \pi_1(f_{n,K})\pi_2(g_n) = \sum_{\mathbf k\in\{0,\ldots,K-1\}^d} \pi_1(h_{\mathbf k,K}) \pi_2(g_{\mathbf k,K}) \,.
	$$
	By Theorem \ref{toric commutator theorem}, we have
	\begin{align*}
&~\quad 	T_K(-\Delta_{\mathbb T^d})^{-\frac12} - A_K \\
& = \sum_{\mathbf k\in\{0,\ldots,K-1\}^d} \pi_1(h_{\mathbf k,K}) \left[\pi_1(h_{\mathbf k,K}),\pi_2(g_{\mathbf k,K})(-\Delta_{\mathbb T^d})^{-\frac12}\right] \in(\mathcal L_{d,\infty})_0 \,.
	\end{align*}
	It follows now from \eqref{eq:torusproof} and Lemma \ref{bs sep lemma} that
	\begin{equation}
		\label{eq:torusproof2}
		\lim_{t\to\infty} t^\frac1d \mu(t,T_K(-\Delta_{\mathbb T^d})^{-\frac12}) = d^{-\frac1d} \left( \frac{1}{K^d} \sum_{\mathbf k\in\{0,\ldots,K-1\}^d} \|g_{\mathbf k,K}\|_d^d \right)^\frac1d.
	\end{equation}

	We note that
	$$
	\frac{1}{K^d} \sum_{\mathbf k\in\{0,\ldots,K-1\}^d} \|g_{\mathbf k,K}\|_d^d = (2\pi)^{-d} \left\|\sum_{\mathbf k\in\{0,\ldots,K-1\}^d} h_{\mathbf k,K}\otimes g_{\mathbf k,K} \right\|_d^d
	$$
	Moreover, it is clear from the expression for $f_{n,K}$ and from the definition of $g_{\mathbf k,K}$ that
	$$\sum_{\mathbf k\in\{0,\ldots,K-1\}^d} h_{\mathbf k,K}\otimes g_{\mathbf k,K}=\sum_{n=1}^{N}f_{n,K}\otimes g_{n}.$$
	Thus, \eqref{eq:torusproof2} is the same as
	\begin{equation}\label{eq:torusproof3}
		\lim_{t\to\infty} t^\frac1d \mu(t,T_K(-\Delta_{\mathbb T^d})^{-\frac12}) = d^{-\frac1d} (2\pi)^{-1} \left\|\sum_{n=1}^{N}f_{n,K}\otimes g_{n} \right\|_d.
	\end{equation}
	
	Note that $f_{n,K}\to f_n$ in the uniform norm as $K\to\infty.$ Hence, $T_K(-\Delta_{\mathbb T^d})^{-\frac12}\to T(-\Delta_{\mathbb T^d})^{-\frac12}$ in $\mathcal{L}_{d,\infty}$ as $K\to\infty$ and $\sum_{n=1}^{N}f_{n,K}\otimes g_{n} \to \sum_{n=1}^{N}f_n\otimes g_n$ in $L_d(\mathbb{T}^d\times\mathbb{S}^{d-1})$ as $K\to\infty$. Theorem \ref{main thm special torus} now follows from \eqref{eq:torusproof3} and Lemma \ref{bs limit lemma}.
\end{proof}


\subsection{Passing from the torus to the whole space}

In this subsection we consider the discrete and continuous Fourier transforms, defined respectively, by
$$
\hat a({\bf t}) = \sum_{n\in\mathbb{Z}^d}a(n)e^{in \cdot {\bf t}} \,,
\qquad {\bf t}\in[-\pi,\pi]^d \,,
$$
and
$$
\mathcal F^{-1} h(\mathbf t) = (2\pi)^{-\frac d2} \int_{\mathbb R^d} h(\xi) e^{i\xi\cdot\mathbf t} \,d\xi \,, \qquad \mathbf t\in\mathbb R^d \,.
$$

\begin{lem}\label{passing to torus lemma} Let $f_1,f_2,h\in L_\infty(\mathbb R^d)$ and suppose that $f_1$ and $f_2$ are supported in $[0,1]^d$ and that $\mathcal{F}^{-1}h$ is a function. If $a\in l_{\infty}(\mathbb{Z}^d)$ satisfies $\hat{a}=(2\pi)^{\frac{d}{2}}\mathcal{F}^{-1}h$ on $[-1,1]^d,$ then
$$\pi_1(f_1)h(-i\nabla)\pi_1(f_2)|_{L_2([0,1]^d)}=M_{f_1}a(-i\nabla_{\mathbb{T}^d})M_{f_2}|_{L_2([0,1]^d)},$$
where, on the right side, $f_j$ is identified with the $2\pi$-periodic extension of $f_j|_{[-\pi,\pi]^d}$.
\end{lem}

\begin{proof} It is clear that $h(-i\nabla)$ is an integral operator on $L_2(\mathbb{R}^d)$ with integral kernel
$$({\bf t},{\bf u})\to (2\pi)^{-\frac{d}{2}}(\mathcal{F}^{-1}h)({\bf t}-{\bf u}),\quad {\bf t},{\bf u}\in\mathbb{R}^d.$$
Thus, $\pi_1(f_1)h(-i\nabla)\pi_1(f_2)$ is an integral operator with integral kernel
$$({\bf t},{\bf u})\to (2\pi)^{-\frac{d}{2}}f_1({\bf t})f_2({\bf u})(\mathcal{F}^{-1}h)({\bf t}-{\bf u}),\quad {\bf t},{\bf u}\in\mathbb{R}^d.$$

On the other hand, $a(-i\nabla_{\mathbb{T}^d})$ is an integral operator on $L_2([-\pi,\pi]^d)$ with integral kernel
$$({\bf t},{\bf u})\to (2\pi)^{-d}\sum_{n\in\mathbb{Z}^d}a(n)e^{in\cdot({\bf t}-{\bf u})}=(2\pi)^{-d}\hat{a}({\bf t}-{\bf u}),\quad {\bf t},{\bf u}\in[-\pi,\pi]^d.$$
Thus, $M_{f_1}a(-i\nabla_{\mathbb{T}^d})M_{f_2}$ is an integral operator on $L_2([-\pi,\pi]^d)$ with integral kernel
$$({\bf t},{\bf u})\to (2\pi)^{-d}f_1({\bf t})f_2({\bf u})\hat{a}({\bf t}-{\bf u}),\quad {\bf t},{\bf u}\in[-\pi,\pi]^d.$$

To prove the equality of those operators, it suffices to establish the equality of their integral kernels. In other words, we need to check
$$(2\pi)^{-\frac{d}{2}}f_1({\bf t})f_2({\bf u})(\mathcal{F}^{-1}h)({\bf t}-{\bf u})=(2\pi)^{-d}f_1({\bf t})f_2({\bf u})\hat{a}({\bf t}-{\bf u}),\quad {\bf t},{\bf u}\in[0,1]^d.$$
This equality follows from the assumption on $\hat{a}$ and the fact that ${\bf t}-{\bf u}\in[-1,1]^d.$
\end{proof}

\begin{fact}\label{denis fact} If distributions $f_1$ and $f_2$ are continuous (except, possibly, at $0$) functions, then so is the distribution $f_1+f_2.$ Furthermore, equality $f_3=f_1+f_2$ holds pointwise (except, possibly, at $0$).
\end{fact}

\begin{lem}\label{homegeneous compute lemma} Let $d>2.$ Let $g\in C^{\infty}(\mathbb{S}^{d-1})$ and set
$$h({\bf t})=g(\frac{{\bf t}}{|{\bf t}|})|{\bf t}|^{-2},\quad {\bf t}\in\mathbb{R}^d.$$
There is an $a\in l_{\infty}(\mathbb{Z}^d)$ such that $\hat{a}=(2\pi)^{\frac{d}{2}}\mathcal{F}^{-1}h$ on $[-1,1]^d$ and such that
$$a({\bf n})=h({\bf n})+O(|{\bf n}|^{-4}),\quad 0\neq {\bf n}\in\mathbb{Z}^d.$$
\end{lem}

\begin{proof} Let $\phi\in C^{\infty}_c(\mathbb{R}^d)$ have support in $[-2,2]^2$ and satisfy $\phi=1$ on $[-1,1]^d.$ Set
$$\hat{a}({\bf t})=(2\pi)^{\frac{d}{2}}\phi({\bf t})\cdot(\mathcal{F}^{-1}h)({\bf t}),\quad {\bf t}\in[-\pi,\pi]^d.$$
We have
\begin{align*}
  a({\bf n})& =(2\pi)^{-d}\int_{[-\pi,\pi]^d}\hat{a}({\bf t})e^{-i {\bf t}\cdot{\bf n}} d{\bf t}\\
&=(2\pi)^{-\frac{d}{2}}\int_{\mathbb{R}^d}\phi({\bf t})(\mathcal{F}^{-1}h)({\bf t})e^{-i {\bf t}\cdot{\bf n}}d{\bf t}=(\mathcal{F}(\phi\cdot\mathcal{F}^{-1}h))({\bf n}).
\end{align*}

Since $h$ is a smooth (except at $0$) homogeneous function of degree $-2,$ it follows from Proposition 2.4.8 and Exercise 2.3.9 (d) in \cite{Grafakos} that $\mathcal{F}^{-1}h$ is a smooth (except at $0$) homogeneous function of degree $2-d.$ By Theorem 2.3.21 in \cite{Grafakos}, $\mathcal{F}(\phi\cdot\mathcal{F}^{-1}h)$ is a smooth function. Using the equality $\mathcal{F}(\mathcal{F}^{-1}h)=h,$ we conclude that
$$\mathcal{F}(\phi\cdot\mathcal{F}^{-1}h)-h=\mathcal{F}((1-\phi)\cdot\mathcal{F}^{-1}h)$$
in the sense of distributions. Since each term on the left side is a continuous (except, possibly, at $0$) function, it follows from Fact \ref{denis fact} that the right hand side is also a continuous (except, possibly, at $0$) function and the equality holds pointwise.

The function $\psi=(1-\phi)\cdot\mathcal{F}^{-1}h$ is smooth. It is clear that $\Delta^2\psi\in L_1(\mathbb{R}^d).$ Denote for brevity $u({\bf t})=|{\bf t}|^4,$ $t\in\mathbb{R}^d.$ We have
$$u\cdot\mathcal{F}\psi=\mathcal{F}(\Delta^2\psi).$$
Thus,
$$\|u\cdot \mathcal{F}\psi\|_{\infty}=\|\mathcal{F}(\Delta^2\psi)\|_{\infty}\leq\|\Delta^2\psi\|_1<\infty.$$
Since $u\cdot \mathcal{F}\psi$ is bounded and since $\mathcal{F}\psi$ is continous (except, possibly, at $0$), it follows that
$$\mathcal{F}\psi({\bf t})=O(|{\bf t}|^{-4}),\quad {\bf t}\in\mathbb{R^d}.$$
In particular,
$$(\mathcal{F}\psi)({\bf n})=O(|{\bf n}|^{-4}),\quad 0\neq{\bf n}\in\mathbb{Z}^d.$$
Thus,
$$(\mathcal{F}(\phi\cdot\mathcal{F}^{-1}h))({\bf n})=h({\bf n})-(\mathcal{F}\psi)({\bf n})=h({\bf n})+O(|{\bf n}|^{-4}),\quad 0\neq{\bf n}\in\mathbb{Z}^d.$$
This completes the proof.
\end{proof}

In the following lemma we slightly abuse the notation and denote
$$(-\Delta_{\mathbb{T}^d})^{-\frac12}\stackrel{def}{=}(-\Delta_{\mathbb{T}^d})^{-\frac12}\cdot (1-P),$$
where $P:L_2(\mathbb{T}^d)\to L_2(\mathbb{T}^d)$ is the orthogonal projection onto the subspace of constants.

\begin{lem}\label{plane minus torus lemma} Let $(f_n)_{n=1}^N\subset L_\infty(\mathbb{R}^d)$ be supported on $[0,1]^d$ and let $(g_n)_{n=1}^N\subset C^{\infty}(\mathbb{S}^{d-1}).$ Then for
$$T=\sum_{n=1}^N\pi_1(f_n)\pi_2(g_n),\quad S=\sum_{n=1}^NM_{f_n}g_n\big(\frac{-i\nabla_{\mathbb{T}^d}}{\sqrt{-\Delta_{\mathbb{T}^d}}}\big).$$
one has
$$
\limsup_{t\to\infty} t^{\frac1d}\mu\Big(t,T(-\Delta)^{-\frac12}\Big) = \limsup_{t\to\infty} t^{\frac1d}\mu\Big(t,S(-\Delta_{\mathbb{T}^d})^{-\frac12}\Big)
$$
and
$$
\liminf_{t\to\infty} t^{\frac1d}\mu\Big(t,T(-\Delta)^{-\frac12}\Big) = \liminf_{t\to\infty} t^{\frac1d}\mu\Big(t,S(-\Delta_{\mathbb{T}^d})^{-\frac12}\Big).
$$
\end{lem}

\begin{proof} Denote for brevity
$$A=T(-\Delta)^{-\frac12},\quad B=S(-\Delta_{\mathbb{T}^d})^{-\frac12}.$$
Then
\begin{align*}
|A^{\ast}|^2 & =\sum_{n_1,n_2=1}^N\pi_1(f_{n_1})h_{n_1,n_2}(-i\nabla)\pi_1(\bar{f}_{n_2}) \,,\\
|B^{\ast}|^2 & =\sum_{n_1,n_2=1}^NM_{f_{n_1}}h_{n_1,n_2}(-i\nabla_{\mathbb{T}^d})M_{\bar{f}_{n_2}} \,,
\end{align*}
where
$$h_{n_1,n_2}({\bf t})=g_{n_1}(\frac{{\bf t}}{|{\bf t}|})\bar{g}_{n_2}(\frac{{\bf t}}{|{\bf t}|})|{\bf t}|^{-2},\quad {\bf t}\in\mathbb{R}^d.$$
Let $a_{n_1,n_2}\in\ell_\infty(\mathbb Z^d)$ such that
$$\hat{a}_{n_1,n_2}=(2\pi)^{\frac{d}{2}}\mathcal{F}^{-1} h_{n_1,n_2}\mbox{ on }[-1,1]^d,$$
and set
$$
X := \sum_{n_1,n_2=1}^NM_{f_{n_1}}a_{n_1,n_2}(-i\nabla_{\mathbb{T}^d})M_{\bar{f}_{n_2}} \,.
$$
Then, by Lemma \ref{passing to torus lemma}, we have
$$|A^{\ast}|^2|_{L_2([0,1]^d)}=X|_{L_2([0,1]^d)},$$
and therefore $\mu(|A^{\ast}|^2|_{L_2([0,1]^d)}) = \mu(X|_{L_2([0,1]^d)})$. Since each $f_n$ is supported on $[0,1]^d$, we have
\begin{align*}
\mu^2(A) = \mu(|A^\ast|^2) = \mu(|A^{\ast}|^2|_{L_2([0,1]^d)}) \
 \,,\qquad
\mu(X) = \mu(X|_{L_2([0,1]^d)}) \,,
\end{align*}
and consequently,
$$
\mu^2(A) = \mu(X) \,.
$$
By Lemma \ref{homegeneous compute lemma}, one can choose $a_{n_1,n_2}$ such that
$$a_{n_1,n_2}({\bf n})=h_{n_1,n_2}({\bf n})+o_{n_1,n_2}({\bf n}),\quad  o_{n_1,n_2}({\bf n})=O(|{\bf n}|^{-4}),\quad 0\neq {\bf n}\in\mathbb{Z}^d.$$
Thus, $X=Y+Z,$
$$Y:= \sum_{n_1,n_2=1}^NM_{f_{n_1}}h_{n_1,n_2}(-i\nabla_{\mathbb{T}^d})M_{\bar{f}_{n_2}},\quad Z := \sum_{n_1,n_2=1}^NM_{f_{n_1}}o_{n_1,n_2}(-i\nabla_{\mathbb{T}^d})M_{\bar{f}_{n_2}}.$$
It is immediate that
$$o_{n_1,n_2}(-i\nabla_{\mathbb{T}^d})\in\mathcal{L}_{\frac{d}{4},\infty}\subset(\mathcal{L}_{\frac{d}{2},\infty})_0.$$
Since $f_n$ is bounded for every $1\leq n\leq N,$ it follows that $Z\in(\mathcal{L}_{\frac{d}{2},\infty})_0.$ Note that $Y=|B^{\ast}|^2.$ Thus,
$$X-|B^{\ast}|^2\in(\mathcal{L}_{\frac{d}{2},\infty})_0.$$
By Lemma \ref{bs sep lemma}, the last relation implies
$$
\limsup_{t\to\infty} s^{\frac2d} \mu(s,X) = \limsup_{t\to\infty} s^{\frac2d} \mu(s,B)^2
$$
and
$$
\liminf_{t\to\infty} s^{\frac2d} \mu(s,X) = \liminf_{t\to\infty} s^{\frac2d} \mu(s,B)^2.
$$
Recalling that $\mu(s,X) = \mu(s,A)^2$, we deduce the assertion of the lemma.
\end{proof}


\subsection{Proof of Theorem \ref{main thm special} for $d>2$}

Combining the results of the previous two subsections, we are finally in position to prove the main result of this secion.

\begin{proof}[Proof of Theorem \ref{main thm special} for $d>2$]
	Suppose that $f_n\in C_c(\mathbb R^d),$ $1\leq n\leq N,$ are supported on $[0,1]^d$ and let $g_n\in C^{\infty}(\mathbb{S}^{d-1}),$ $1\leq n\leq N.$ Then, identifying $f_n$ with the $2\pi$-periodic extension of $f_n|_{[-\pi,\pi]^d}$, we deduce from Theorem \ref{main thm special torus} that
$$
\lim_{t\to\infty}t^{\frac1d}\mu\Big(t,\Big(\sum_{n=1}^NM_{f_n}g_n\big(\frac{-i\nabla_{\mathbb{T}^d}}{\sqrt{-\Delta_{\mathbb{T}^d}}}\big)\Big)(-\Delta_{\mathbb{T}^d})^{-\frac12}\Big)
= d^{-\frac1d} (2\pi)^{-1} \left\| \sum_{n=1}^N f_n\otimes g_n \right\|_d .
$$
By Lemma \ref{plane minus torus lemma} and the fact that
$$
\left\| \sum_{n=1}^N f_n\otimes g_n \right\|_d = \|{\rm sym}(T)\|_d \,,
$$
this implies that
$$
\lim_{t\to\infty}t^{\frac1d}\mu\Big(t,\Big(\sum_{n=1}^N\pi_1(f_n)\pi_2(g_n)\Big)(-\Delta)^{-\frac12}\Big) = d^{-\frac1d} (2\pi)^{-1} \|{\rm sym}(T)\|_d \,.
$$
By applying dilations and translations, we can remove the restriction on the support of $f_n$'s. By Lemmas \ref{d>2 useful cwikel} and \ref{bs limit lemma}, we can remove the restriction on the smoothness of the $g_n$'s. 

The mapping
$$t\to |t|^{-1}-(1+|t|^2)^{-\frac12},\quad t\in\mathbb{R}^d,$$
belongs to $(L_{d,\infty}(\mathbb{R}^d))_0.$ By \cite[Subsection 5.3]{BKS} (see also \cite[Corollary 1.2]{LeSZ}), we have
$$\pi_1(f_n)(-\Delta)^{-\frac12}-\pi_1(f_n)(1-\Delta)^{-\frac12}\in(\mathcal{L}_{d,\infty})_0,\quad 1\leq n\leq N,$$
and, therefore,
$$\sum_{n=1}^N\pi_1(f_n)\pi_2(g_n)(-\Delta)^{-\frac12}-\sum_{n=1}^N\pi_1(f_n)\pi_2(g_n)(1-\Delta)^{-\frac12}\in(\mathcal{L}_{d,\infty})_0.$$
Hence, by Lemma \ref{bs sep lemma}, we can replace $-\Delta$ with $1-\Delta.$ This completes the proof of Theorem \ref{main thm special} for $d>2$.
\end{proof}


\section{Proof of Theorem \ref{main thm} in a special case. Case $d=2$}

In this section we prove Theorem \ref{main thm special} for $d=2$.


\subsection{The set $\mathcal X_2$}

\begin{defi}\label{x2 def} Let $\mathcal{X}_2$ be the set of all compactly supported $f\in L_{\infty}(\mathbb{R}^2)$ such that, for every Borel set $I\subset\mathbb{S}^1,$ one has
$$
\lim_{t\to\infty}t^{\frac12}\mu\Big(t,\pi_1(f)\pi_2(\chi_I)(1-\Delta)^{-\frac12}\Big)=2^{-\frac12} (2\pi)^{-1} m(I)^{\frac12}\|f\|_2.
$$
\end{defi}

In this subsection, we frequently use the unitary operator ${\rm Shift}_c:L_2(\mathbb{R}^2)\to L_2(\mathbb{R}^2),$ $c\in\mathbb{R}^2,$ defined by setting
$$({\rm Shift}_cf)(u)=f(u+c),\quad f\in L_2(\mathbb{R}^2).$$ 
We also frequently use the dilation operator $\sigma_t,$ $0<t\in\mathbb{R},$ defined by setting
$$(\sigma_tf)(u)=f(\frac{u}{t}),\quad f\in L_2(\mathbb{R}^2).$$

\begin{lem}\label{d=2 preliminary lemma} Let $(f_l)_{l\geq0}\in\mathcal{X}_2.$
\begin{enumerate}
\item\label{dpla} if $(f_l)_{l=0}^n$ are pairwise disjointly supported, then $\sum_{l=0}^nf_l\in\mathcal{X}_2;$
\item\label{dplb} if $c_l\in\mathbb{R}^2,$ then ${\rm Shift}_{c_l}(f_l)\in\mathcal{X}_2;$
\item\label{dplc} if $t>0,$ then $\sigma_tf_l\in\mathcal{X}_2;$
\item\label{dpld} if $(f_l)_{l\geq0}$ are supported on a compact set $K$ and $f_l\to f\in L_{\infty}(\mathbb{R}^2)$ in $(L_2\log L)(\mathbb{R}^2),$ then $f\in\mathcal{X}_2;$
\end{enumerate}	
\end{lem}

To be more precise, by the assumption that $f_l$ and $f_{l'}$ are pairwise disjointly supported we mean that
$$
m( {\rm supp} f_l \cap {\rm supp} f_{l'}) =0 \,.
$$

\begin{proof}
	Fix a Borel set $I\subset\mathbb{S}^1$.
	
	For the proof of the first assertion we want to apply Lemma \ref{second abstract limit lemma} and set
$$A=\pi_1(\sum_{l=0}^nf_l)\pi_2(\chi_I)(1-\Delta)^{-\frac12} \,.
$$
Since $\pi_2(\chi_I)$ commutes with $(1-\Delta)^{-\frac12}$ and $\sum_{l=0}^nf_l$ is a  bounded, compactly supported function, we deduce from Theorem \ref{solomyak thm} that $A \in \mathcal L_{2,\infty}$. Let
$$p_l=\pi_1(\chi_{{\rm supp}(f_l)}),\quad 0\leq l\leq n,$$
and let $p_{-1}=1-\sum_{0\leq l<k} p_l$.
The first assumption in Lemma \ref{second abstract limit lemma} follows from the equality
$$p_lA=\pi_1(f_l)\pi_2(\chi_I)(1-\Delta)^{-\frac12},\quad 0\leq l\leq n,$$
and the assumption $f_l\in\mathcal{X}_2,$ $0\leq l\leq n$, as well as $p_{-1}A=0$. Since $\pi_1(\sum_{k=0}^nf_k)$ commutes with $p_l,$ $-1\leq l\leq n,$ the second assumption in Lemma \ref{second abstract limit lemma} follows from Theorem \ref{commutator theorem}. Applying Lemma \ref{second abstract limit lemma}, we complete the proof of the first assertion.

It is immediate that
$$\pi_1({\rm Shift}_{c_l}(f_l))\pi_2(\chi_I)(1-\Delta)^{-\frac12}={\rm Shift}_{c_l}\cdot \pi_1(f_l)\pi_2(\chi_I)(1-\Delta)^{-\frac12}\cdot {\rm Shift}_{-c_l}.$$
Therefore,
$$\mu(\pi_1({\rm Shift}_{c_l}(f_l))\pi_2(\chi_I)(1-\Delta)^{-\frac12})=\mu(\pi_1(f_l)\pi_2(\chi_I)(1-\Delta)^{-\frac12})$$
The second assertion follows from the latter equality and the assumption $f_l\in\mathcal{X}_2.$

To see the third assertion, note that
$$\pi_1(\sigma_tf_l)\pi_2(\chi_I)(1-\Delta)^{-\frac12}=\sigma_t\cdot \pi_1(f_l)\pi_2(\chi_I)(1-t^{-2}\Delta)^{-\frac12}\cdot\sigma_t^{-1}.$$
Thus,
$$\mu(\pi_1(\sigma_tf_l)\pi_2(\chi_I)(1-\Delta)^{-\frac12})=\mu(\pi_1(f_l)\pi_2(\chi_I)(1-t^{-2}\Delta)^{-\frac12}).$$
Note that, by a straightforward evaluation of the Hilbert--Schmidt norm,
$$\pi_1(f_l)\pi_2(\chi_I)(1-t^{-2}\Delta)^{-\frac12}-t\pi_1(f_l)\pi_2(\chi_I)(1-\Delta)^{-\frac12}\in\mathcal{L}_2\subset(\mathcal{L}_{2,\infty})_0.$$
Using $f_l\in\mathcal X_2$ and Lemma \ref{bs sep lemma}, we deduce
$$\lim_{t\to\infty}t^{\frac12}\mu(\pi_1(\sigma_tf_l)\pi_2(\chi_I)(1-\Delta)^{-\frac12})=t\cdot m(I)^{\frac12}\|f_l\|_2=m(I)^{\frac12}\|\sigma_tf_l\|_2.$$

The fourth assertion follows from Theorem \ref{solomyak thm} and Lemma \ref{bs limit lemma}.
\end{proof}

The following lemma contains the main idea of the proof of Theorem \ref{main thm special} for $d=2$. The rest of the argument consists either of verification of its conditions or of its applications.

\begin{lem}\label{d=2 radial lemma}
	If $f\in L_\infty(\mathbb R^2)$ is compactly supported and rotation invariant, then $f\in \mathcal{X}_2.$
\end{lem}	
\begin{proof}
	Set
	$$A=\pi_1(f)(1-\Delta)^{-\frac12} \,.$$
	We claim that
	$$\lim_{t\to\infty}t^{\frac12}\mu(t,A)=(4\pi)^{-\frac12} \|f\|_2.$$
	Indeed, for $f\in C_c(\mathbb R^2)$ (not necessarily rotation invariant) this follows by the Birman--Schwinger principle and simple Dirichlet--Neumann bracketing for the resulting Schr\"odinger operator $-\Delta+1-\alpha f^2$ as in the proof of \cite[Theorem 4.28]{FrLaWe}. Using Theorem \ref{solomyak thm} and Lemma \ref{bs limit lemma}, the asymptotics extend to all compactly supported $f\in (L_2\log L)(\mathbb R^2)$, in particular, to those in the statement of the lemma.
	
	To deduce the assertion of the lemma we want to apply Lemma \ref{third abstract limit lemma} with $p=2$. We define a spectral measure $\nu$ on the Borel $\sigma$-algebra on $[0,1)$ by setting
	$$\nu(I)=\pi_2(\chi_{2\pi I{\rm mod}2\pi}),\quad I\subset[0,1).$$
	That $\nu$ is indeed a spectral measure follows from the continuity of $\pi_2$ in the strong operator topology.
	
	Let $R_{\theta}$ be a rotation operator on $L_2(\mathbb{R}^2)$ by the angle $\theta\in(0,2\pi).$ If $I_1$ and $I_2$ are sub-intervals in $[0,1)$ of equal length, then there is a $\theta$ such that
	$$\nu(I_2)=R_{\theta}^{-1}\nu(I_1)R_{\theta}.$$
	Since $f$ is rotation invariant, it follows that
	$$A=R_{\theta}^{-1}AR_{\theta}.$$
	This verifies the first assumption on $\nu$ in Lemma \ref{third abstract limit lemma}.
	
The second assumption on $\nu$ in Lemma \ref{third abstract limit lemma} follows from Theorem \ref{commutator theorem}. The third assumption on $\nu$ in Lemma \ref{third abstract limit lemma} follows from Lemma \ref{d=2 useful limsup} with $g=\chi_{2\pi I{\rm mod}2\pi}$ and with $\psi(t)=C t^{\frac14},$ $t\in(0,1).$	Thus, all the assumptions in Lemma \ref{third abstract limit lemma} are met. Applying this lemma, we complete the proof.
\end{proof}

\begin{lem}\label{d=2 main lemma} Every $f\in C_c(\mathbb{R}^2)$ belongs to $\mathcal{X}_2.$
\end{lem}	
\begin{proof} By applying dilation and translation, we may assume without loss of generality that $f$ is supported on $[0,1]^2.$
	
The idea is to approximate $f$ in $(L_2\log L)(\mathbb{R}^2)$ by a sum of pairwise disjointly supported (shifted) rotation invariant functions.

For $l\in\mathbb{N},$ choose a sequence $\{B_{r_{m,l}}(c_{m,l})\}_{m\geq0}$ of non-intersecting balls such that each radius $r_{m,l}<\frac1l$ and such that
$$\bigcup_{m\geq0}B_{r_{m,l}}(c_{m,l})=[0,1]^2$$
modulo a set of measure $0.$

Set
$$f_{m,l}= \chi_{B_{r_{m,l}}(c_{m,l})}\cdot \frac1{{\rm Vol}(B_{r_{m,l}}(c_{m,l})}\int_{B_{r_{m,l}}(c_{m,l})}f.$$
Clearly, ${\rm Shift}_{-c_{m,l}}f_{m,l}$ is bounded, compactly supported and rotation invariant, so, by Lemma \ref{d=2 radial lemma}, ${\rm Shift}_{-c_{m,l}}f_{m,l} \in\mathcal{X}_2.$ Thus, by Lemma \ref{d=2 preliminary lemma} \eqref{dplb}, $f_{m,l}\in\mathcal{X}_2.$ For a fixed $l\in\mathbb{N},$ the functions $(f_{m,l})_{m\geq 0}$ are pairwise disjointly supported. By Lemma \ref{d=2 preliminary lemma} \eqref{dpla}, we have $\sum_{m=0}^M f_{m,l}\in\mathcal{X}_2$ for every $M\geq0$. By Lemma \ref{d=2 preliminary lemma} \eqref{dpld}, we have $\sum_{m\geq0} f_{m,l}\in\mathcal{X}_2$ for every fixed $l\in\mathbb{N}.$

Since $f\in C([0,1]^2),$ it follows that
$$\sum_{m\geq0} f_{m,l}\to f,\quad l\to\infty,$$
in the uniform norm. Applying again Lemma \ref{d=2 preliminary lemma} \eqref{dpld}, we complete the proof.
\end{proof}


\subsection{Proof of Theorem \ref{main thm special} for $d=2$}

As a consequence of the results in the previous subsection, we are finally in position to complete the proof of Theorem \ref{main thm special} in the remaining case $d=2$.

\begin{proof}[Proof of Theorem \ref{main thm special} for $d=2$]
	For every $K\in\mathbb{N}$ and $0\leq k<K$, let $h_{k,K}(s)=\chi_{[\frac{2\pi k}{K},\frac{2\pi (k+1)}{K})}({\rm Arg}(s)),$ $s\in\mathbb{S}^1.$ Let $g_{n,K}$ be the conditional expectation on $L_{\infty}(\mathbb{S}^1)$ of $g_n$ onto the subalgebra generated by $(h_{k,K})_{0\leq k<K}$ and write
$$g_{n,K}=\sum_{k=0}^{K-1}c_{n,k,K}h_{k,K}.$$
We set
$$f_{k,K}=\sum_{n=1}^{N}c_{n,k,K}f_n$$
and
$$S_K=\sum_{k=0}^{K-1}\pi_2(h_{k,K})\pi_1(f_{k,K})\pi_2(h_{k,K}) \,.$$

The operators
$$\pi_2(h_{k,K})\pi_1(f_{k,K})\pi_2(h_{k,K})(1-\Delta)^{-\frac12},\quad 0\leq k<K,$$
are pairwise orthogonal. Thus,
$$\mu(S_K(1-\Delta)^{-\frac12})=\mu\Big(\bigoplus_{k=0}^{K-1}\pi_2(h_{k,K})\pi_1(f_{k,K})\pi_2(h_{k,K})(1-\Delta)^{-\frac12}\Big).$$
By Lemma \ref{d=2 main lemma} and Definition \ref{x2 def}, we have
$$
\lim_{t\to\infty}t^{\frac12}\mu(t,\pi_1(f_{k,K})\pi_2(h_{k,K})(1-\Delta)^{-\frac12}) =  (4\pi)^{-\frac12} K^{-\frac12} \|f_{k,K}\|_2 \,.
$$
Meanwhile, by Theorem \ref{commutator theorem},
\begin{align*}
&~\quad \pi_2(h_{k,K})\pi_1(f_{k,K})\pi_2(h_{k,K})(1-\Delta)^{-\frac12}  - \pi_1(f_{k,K})\pi_2(h_{k,K})(1-\Delta)^{-\frac12} \\
&= [\pi_2(h_{k,K}), \pi_1(f_{k,K})(1-\Delta)^{-\frac12}] \pi_2(h_{k,K}) \in (\mathcal L_{2,\infty})_0 \,,
\end{align*}
so, by Lemma \ref{bs sep lemma},
$$
\lim_{t\to\infty}t^{\frac12}\mu(t,\pi_2(h_{k,K})\pi_1(f_{k,K})\pi_2(h_{k,K})(1-\Delta)^{-\frac12}) = (4\pi)^{-\frac12} K^{-\frac12} \|f_{k,K}\|_2 \,.
$$
We deduce, by Lemma \ref{bs pairwise orthogonal lemma},
\begin{align*}
& ~\quad \lim_{t\to\infty}t^{\frac12}\mu\Big(t,\bigoplus_{k=0}^{K-1}\pi_2(h_{k,K})\pi_1(f_{k,K})\pi_2(h_{k,K})(1-\Delta)^{-\frac12}\Big) \\
 &=(4\pi)^{-\frac12} \Big( \frac{1}K \sum_{k=0}^{K-1} \|f_{k,K}\|_2^2\Big)^{\frac12}.
\end{align*}
In other words, we have
\begin{equation}\label{dfl eq0}
\lim_{t\to\infty}t^{\frac12}\mu(S_K(1-\Delta)^{-\frac12})
=(4\pi)^{-\frac12} \Big( \frac{1}K \sum_{k=0}^{K-1}\|f_{k,K}\|_2^2\Big)^{\frac12}.
\end{equation}

Now let us introduce
$$T_K=\sum_{n=1}^{N}\pi_1(f_n)\pi_2(g_{n,K})=\sum_{k=0}^{K-1}\pi_1(f_{k,K})\pi_2(h_{k,K}) \,.$$
By Theorem \ref{commutator theorem}, we have
$$(T_K-S_K)(1-\Delta)^{-\frac12}=\sum_{k=0}^{K-1}[\pi_1(f_{k,K})(1-\Delta)^{-\frac12},\pi_2(h_{k,K})]\cdot \pi_2(h_{k,K})\in(\mathcal{L}_{2,\infty})_0.$$
It follows now from \eqref{dfl eq0} and Lemma \ref{bs sep lemma} that
\begin{equation}\label{dfl eq1}
\lim_{t\to\infty}t^{\frac12}\mu(t,T_K(1-\Delta)^{-\frac12})= (4\pi)^{-\frac12} \Big( \frac{1}K \sum_{k=0}^{K-1}\|f_{k,K}\|_2^2\Big)^{\frac12}.
\end{equation}

We note that
$$
\frac{1}K \sum_{k=0}^{K-1}\|f_{k,K}\|_2^2 = (2\pi)^{-1} \left\|\sum_{k=0}^{K-1}f_{k,K}\otimes h_{k,K} \right\|_2^2 \,.
$$
Moreover, it is clear from the expression for $g_{n,K}$ and from the definition of $f_{k,K}$ that
$$\sum_{k=0}^{K-1}f_{k,K}\otimes h_{k,K}=\sum_{k=0}^{K-1}\sum_{n=1}^Nc_{n,k,K}f_n\otimes h_{k,K}=\sum_{n=1}^{N}f_n\otimes g_{n,K}.$$
Thus, \eqref{dfl eq1} is the same as
\begin{equation}\label{dfl eq2}
\lim_{t\to\infty}t^{\frac12}\mu(T_K(1-\Delta)^{-\frac12})= 2^{-\frac12} (2\pi)^{-1} \left\| \sum_{n=1}^{N}f_n\otimes g_{n,K} \right\|_2.
\end{equation}

Note that $g_{n,K}\to g_n$ in the uniform norm as $K\to\infty.$ Hence, $T_K(1-\Delta)^{-\frac12}\to T(1-\Delta)^{-\frac12}$ in $\mathcal{L}_{2,\infty}$ as $K\to\infty$ and $\sum_{n=1}^{N}f_n\otimes g_{n,K}\to \sum_{n=1}^{N}f_n\otimes g_n$ in $L_2(\mathbb{R}^2\times\mathbb{S}^1)$ as $K\to\infty$. The assertion follows now from \eqref{dfl eq2} and Lemma \ref{bs limit lemma}.
\end{proof}


\section{Proof of Theorem \ref{main thm}}

In this section, we prove our main result on spectral asymptotics, Theorem \ref{main thm}. A crucial ingredient will be Theorem \ref{main thm special}. In order to deduce the former from the latter, we will use the following two simple lemmas.

\begin{lem}\label{commutator is compact} If $T_1,T_2\in\Pi,$ then
$$[T_1,T_2]\in \mathcal{K}(L_2(\mathbb{R}^d)).$$
\end{lem}
\begin{proof} Let $q:B(L_2(\mathbb{R}^d))\to B(L_2(\mathbb{R}^d))/\mathcal{K}(L_2(\mathbb{R}^d))$ be the canonical quotient map. Recall (see the proof of \cite[Theorem 3.3]{DAO2}) that ${\rm sym}$ is constructed as a composition
$${\rm sym}=\theta^{-1}\circ q,$$
where $\theta:\mathcal{A}_1\otimes_{{\rm min}}\mathcal{A}_2\to q(\Pi)$ is some $\ast$-isomorphism (its definition and properties are irrelevant at the current proof). Since ${\rm sym}$ is a $\ast$-homomorphism, it follows that
$${\rm sym}([T_1,T_2])=[{\rm sym}(T_1),{\rm sym}(T_2)]=0,\quad T_1,T_2\in\Pi.$$
Thus,
$$
q([T_1,T_2])= \theta({\rm sym}([T_1,T_2]) = 0 \,,
$$
which yields the assertion.
\end{proof}

\begin{lem}\label{ber request lemma1} Let $(f_k)_{k=1}^m\subset\mathcal{A}_1$ and $(g_k)_{k=1}^m\subset\mathcal{A}_2.$ Then
$$\prod_{k=1}^m\pi_1(f_k)\pi_2(g_k)\in\pi_1(\prod_{k=1}^mf_k)\pi_2(\prod_{k=1}^mg_k)+\mathcal{K}(L_2(\mathbb{R}^d)).$$
\end{lem}
\begin{proof} We prove the assertion by induction on $m.$ For $m=1,$ there is nothing to prove. Let us prove the assertion for $m=2.$ We have
\begin{align*}
&~\quad \pi_1(f_1)\pi_2(g_1)\pi_1(f_2)\pi_2(g_2)\\
&=[\pi_2(g_1),\pi_1(f_1f_2)]\cdot \pi_2(g_2)
+[\pi_1(f_1),\pi_2(g_1)]\cdot \pi_1(f_2)\pi_2(g_2)+\pi_1(f_1f_2)\pi_2(g_1g_2).
\end{align*}
By Lemma \ref{commutator is compact}, we have
$$[\pi_1(f_1),\pi_2(g_1)],[\pi_2(g_1),\pi_1(f_1f_2)]\in \mathcal{K}(L_2(\mathbb{R}^d)).$$
Therefore,
$$\pi_1(f_1)\pi_2(g_1)\pi_1(f_2)\pi_2(g_2)\in \pi_1(f_1f_2)\pi_2(g_1g_2)+\mathcal{K}(L_2(\mathbb{R}^d)).$$
This proves the assertion for $m=2.$

It remains to prove the step of induction. Suppose the assertion holds for $m\geq 2$ and let us prove it for $m+1.$ Since
$$\prod_{k=1}^{m+1}\pi_1(f_k)\pi_2(g_k)=\pi_1(f_1)\pi_2(g_1)\cdot\prod_{k=2}^{m+1}\pi_1(f_k)\pi_2(g_k),$$
using the inductive assumption, we obtain
$$\prod_{k=1}^{m+1}\pi_1(f_k)\pi_2(g_k)\in\pi_1(f_1)\pi_2(g_1)\cdot\pi_1(\prod_{k=2}^{m+1}f_k)\pi_2(\prod_{k=2}^{m+1}g_k)+\mathcal{K}(L_2(\mathbb{R}^d)).$$
Using the assertion for $m=2,$ we obtain
$$\pi_1(f_1)\pi_2(g_1)\cdot\pi_1(\prod_{k=2}^{m+1}f_k)\pi_2(\prod_{k=2}^{m+1}g_k)\in\pi_1(\prod_{k=1}^{m+1}f_k)\pi_2(\prod_{k=1}^{m+1}g_k)+\mathcal{K}(L_2(\mathbb{R}^d)).$$
Combining the last two equations, we obtain
$$\prod_{k=1}^{m+1}\pi_1(f_k)\pi_2(g_k)\in\pi_1(\prod_{k=1}^{m+1}f_k)\pi_2(\prod_{k=1}^{m+1}g_k)+\mathcal{K}(L_2(\mathbb{R}^d)).$$
This establishes the step of induction and completes the proof of the lemma.
\end{proof}

We are finally in position to prove our main result on spectral asymptotics.

\begin{proof}[Proof of Theorem \ref{main thm}] By the definition of the $C^{\ast}$-algebra $\Pi,$ for every $T\in\Pi,$ there is a sequence $(T_n)_{n\geq1}$ in the $\ast$-algebra generated by $\pi_1(\mathcal{A}_1)$ and $\pi_2(\mathcal{A}_2)$ such that $T_n\to T$ in the uniform norm. We can write
	$$T_n=\sum_{l=1}^{l_n}\prod_{k=1}^{k_n}\pi_1(f_{n,k,l})\pi_2(g_{n,k,l})$$
	with $f_{n,k,l}\in\mathcal A_1$ and $g_{n,k,l}\in\mathcal A_2$. Let us abbreviate
	$$f_{n,l}=\prod_{k=1}^{k_n}f_{n,k,l}\in\mathcal{A}_1,\quad g_{n,l}=\prod_{k=1}^{k_n}g_{n,k,l}\in\mathcal{A}_2.$$
Then, by Lemma \ref{ber request lemma1}, we have
$$
S_n = T_n - \sum_{l=1}^{l_n}\pi_1(f_{n,l})\pi_2(g_{n,l}) \in\mathcal{K}(L_2(\mathbb{R}^d)).
$$

By assumption in Theorem \ref{main thm}, the operator $T$ is compactly supported. Hence, $T=T\pi_1(\phi)$ for some $\phi\in C^{\infty}_c(\mathbb{R}^d).$ We decompose
$$T_n\pi_1(\phi)=A_n+B_n+C_n,$$
where
$$A_n=\sum_{l=1}^{l_n}\pi_1(\phi\cdot f_{n,l})\pi_2(g_{n,l}),$$
$$B_n=\sum_{l=1}^{l_n}\pi_1(f_{n,l})\cdot [\pi_2(g_{n,l}),\pi_1(\phi)],$$
$$C_n=S_n\pi_1(\phi).$$

By Theorem \ref{main thm special}, we have
\begin{equation*}
\lim_{t\to\infty}t^{\frac1d}\mu(t,A_n(1-\Delta)^{-\frac12})=d^{-\frac1d} (2\pi)^{-1} \|{\rm sym}(A_n)\|_d.
\end{equation*}
By Theorem \ref{commutator theorem}, we have $B_n(1-\Delta)^{-\frac12}\in(\mathcal{L}_{d,\infty})_0.$ Since $S_n$ is compact and since $\pi_1(\phi)(1-\Delta)^{-\frac12}\in\mathcal{L}_{d,\infty},$ it follows that $C_n(1-\Delta)^{-\frac12}\in(\mathcal{L}_{d,\infty})_0.$ Consequently, $(B_n+C_n)(1-\Delta)^{-\frac12}\in(\mathcal{L}_{d,\infty})_0$ and, by Lemma \ref{bs sep lemma},
$$
\lim_{t\to\infty}t^{\frac1d}\mu(t,(A_n+B_n+C_n)(1-\Delta)^{-\frac12}) = \lim_{t\to\infty}t^{\frac1d}\mu(t,A_n(1-\Delta)^{-\frac12}) \,.
$$
On the other hand, since $B_n$ and $C_n$ are compact (the former by Lemma \ref{commutator is compact}) and since ${\rm sym}$ vanishes on every compact operator in $\Pi$ (see the proof of Lemma \ref{commutator is compact}), it follows that
$$
{\rm sym}(A_n)={\rm sym}(A_n+B_n+C_n) \,.
$$
Combining all these assertions, we find that
\begin{equation*}
\lim_{t\to\infty}t^{\frac1d}\mu(t,A_n+B_n+C_n)=d^{-\frac1d}  (2\pi)^{-1} \|{\rm sym}(A_n+B_n+C_n)\|_d \,.
\end{equation*}

Taking into account that $A_n+B_n+C_n=T_n\pi_1(\phi),$ we rewrite these asymptotics as
\begin{equation}\label{mt eq2}
\lim_{t\to\infty}t^{\frac1d}\mu(t,T_n\pi_1(\phi))= d^{-\frac1d} (2\pi)^{-1} \|{\rm sym}(T_n\pi_1(\phi))\|_d.
\end{equation}
Since $T_n\pi_1(\phi)\to T\pi_1(\phi)=T$ in the uniform norm, we deduce, on the one hand, using dominated convergence that
$$
\|{\rm sym}(T_n\pi_1(\phi))\|_d \to \|{\rm sym}(T)\|_d
$$
and, on the other hand, using the fact\footnote{For $d=2$ this follows from Theorem \ref{solomyak thm}. For $d>2,$ this follows from Cwikel's estimate in $\mathcal{L}_{d,\infty}.$} that $\pi_1(\phi)(1-\Delta)^{-\frac12}\in\mathcal L_{d,\infty}$, that
$$T_n\pi_1(\phi)(1-\Delta)^{-\frac12}\to T(1-\Delta)^{-\frac12}
\qquad\text{in}\ \mathcal L_{d,\infty} \,.
$$
The assertion follows now from \eqref{mt eq2} and Lemma \ref{bs limit lemma}.	
\end{proof}


\section{Proof of Theorems \ref{mainiff} and \ref{main cor}}

Our goal in this section is to prove Theorems \ref{mainiff} and \ref{main cor} concerning
$$\dbar f = i [{\rm sgn}\,\mathcal D,1\otimes M_f].$$

We begin with one implication in Theorem \ref{mainiff}, which is a simple consequence of previous work in \cite{LMSZ}.

\begin{prop}\label{upperbound}
If $d\geq 2$ and if $f\in\dot W^{1}_d(\mathbb{R}^d),$ then $\dbar f\in\mathcal L_{d,\infty}$ with
$$\|\dbar f\|_{d,\infty} \leq C_d \|\nabla f\|_d.$$
\end{prop}

\begin{proof}
	We know \cite[Theorem 11.43]{Leoni-book} that there is a sequence $(f_n)\subset C^\infty_c(\mathbb R^d)$ with $\nabla f_n\to\nabla f$ in $L_d$. Applying \cite[Theorem 1]{LMSZ} we infer that
	$$
	\| \dbar f_n \|_{d,\infty} \leq C_d \|\nabla f_n \|_{d} \,.
	$$
	The continuity of the embedding $\dot W^1_d\subset BMO$ implies that $f_n\to f$ in $BMO$ and therefore, by \cite[Theorem I]{CoRoWe}, $\dbar f_n \to\dbar f$ in uniform norm. By the Fatou property, this implies $\dbar f\in\mathcal L_{d,\infty}$ and
	$$
	\|\dbar f \|_{d,\infty} \leq \liminf_{n\to\infty} \| \dbar f_n \|_{d,\infty} \leq C_d \liminf_{n\to\infty} \|\nabla f_n \|_{d} = C_d \|\nabla f\|_{d} \,,
	$$
	proving the claim.
\end{proof}

The converse direction in Theorem \ref{mainiff} is harder. We begin by computing the asymptotics in the smooth case. We make use of the following simple lemma.

\begin{lem}\label{ak properties lemma} Let $d\geq 2$, let $f\in C^\infty_c(\mathbb R^d)$ be real-valued and set
$$A_k=M_{\partial_kf}-\sum_{j=1}^d\frac{D_kD_j}{-\Delta}M_{\partial_jf},\quad 1\leq k\leq d.$$
Then one has
\begin{equation}
(\Im A_k)(1-\Delta)^{-\frac12}\in(\mathcal{L}_{d,\infty})_0,\quad 1\leq k\leq d,
\end{equation}
\begin{equation}
(1-\Delta)^{-\frac12}[A_k,A_l](1-\Delta)^{-\frac12}\in(\mathcal{L}_{\frac{d}{2},\infty})_0,\quad 1\leq k,l\leq d.
\end{equation}
\end{lem}
\begin{proof} Since $f$ is real-valued, it follows that
$$A_k^{\ast}=M_{\partial_kf}-\sum_{j=1}^dM_{\partial_jf}\frac{D_kD_j}{-\Delta},\quad 1\leq k\leq d.$$
Thus,
$$\Im A_k =-\sum_{j=1}^d\big[\frac{D_kD_j}{-\Delta},M_{\partial_jf}\big],\quad 1\leq k\leq d.$$
The first inclusion follows now from Theorem \ref{commutator theorem} (we apply it to the function $g({\bf s})= s_ks_j,$ ${\bf s}\in\mathbb{S}^{d-1}$).

Next,
\begin{align*}
[A_k,A_l]&=-\sum_{j_2=1}^d[M_{\partial_k f},\frac{D_lD_{j_2}}{-\Delta}M_{\partial_{j_2}f}]-\sum_{j_1=1}^d[\frac{D_kD_{j_1}}{-\Delta}M_{\partial_{j_1}f},M_{\partial_lf}]\\
&\quad +\sum_{j_1,j_2=1}^d[\frac{D_kD_{j_1}}{-\Delta}M_{\partial_{j_1}f},\frac{D_lD_{j_2}}{-\Delta}M_{\partial_{j_2}f}]\\
&=-\sum_{j_2=1}^d[M_{\partial_k f},\frac{D_lD_{j_2}}{-\Delta}]\cdot M_{\partial_{j_2}f}-\sum_{j_1=1}^d[\frac{D_kD_{j_1}}{-\Delta},M_{\partial_lf}]\cdot M_{\partial_{j_1}f}\\
&\quad +\sum_{j_1,j_2=1}^d\frac{D_kD_{j_1}}{-\Delta}\cdot [M_{\partial_{j_1}f},\frac{D_lD_{j_2}}{-\Delta}]\cdot M_{\partial_{j_2}f}\\
&\qquad +\sum_{j_1,j_2=1}^d\frac{D_lD_{j_2}}{-\Delta}\cdot [\frac{D_kD_{j_1}}{-\Delta},M_{\partial_{j_2}f}]\cdot M_{\partial_{j_1}f}.
\end{align*}
Therefore,
\begin{align*}
&~\quad (1-\Delta)^{-\frac12}[A_k,A_l](1-\Delta)^{-\frac12}\\
&=-\sum_{j_2=1}^d[(1-\Delta)^{-\frac12}M_{\partial_k f},\frac{D_lD_{j_2}}{-\Delta}]\cdot M_{\partial_{j_2}f}(1-\Delta)^{-\frac12}\\
&~\quad -\sum_{j_1=1}^d[\frac{D_kD_{j_1}}{-\Delta},(1-\Delta)^{-\frac12}M_{\partial_lf}]\cdot M_{\partial_{j_1}f}(1-\Delta)^{-\frac12}\\
&~\quad 
+\sum_{j_1,j_2=1}^d\frac{D_kD_{j_1}}{-\Delta}\cdot [(1-\Delta)^{-\frac12}M_{\partial_{j_1}f},\frac{D_lD_{j_2}}{-\Delta}]\cdot M_{\partial_{j_2}f}(1-\Delta)^{-\frac12}\\
&~\quad  +\sum_{j_1,j_2=1}^d\frac{D_lD_{j_2}}{-\Delta}\cdot [\frac{D_kD_{j_1}}{-\Delta},(1-\Delta)^{-\frac12}M_{\partial_{j_2}f}]\cdot M_{\partial_{j_1}f}(1-\Delta)^{-\frac12}.
\end{align*}
By Theorem \ref{commutator theorem}, each commutator factor in the above formula belongs to $(\mathcal{L}_{d,\infty})_0.$ Since $M_{\partial_jf}(1-\Delta)^{-\frac12}\in\mathcal{L}_{d,\infty},$ $1\leq j\leq d,$ the assertion follows now from H\"older's inequality.
\end{proof}

The following important lemma explains the role of the operator $A$.

\begin{lem}\label{dbarf abs compute lemma} Let $d\geq 2$ and let $f\in C^\infty_c(\mathbb R^d)$ be real-valued. Then, with $A_k$ defined in Lemma \ref{ak properties lemma}, we have
$$|\dbar f|^2-1\otimes (1-\Delta)^{-\frac12}(\sum_{k=1}^d|A_k|^2)(1-\Delta)^{-\frac12}\in(\mathcal{L}_{\frac{d}{2},\infty})_0.$$
\end{lem}

\begin{proof} We have (see \cite[Theorem 6.3.1]{LMSZ-book})
$$\dbar f - A(1+\mathcal D^2)^{-\frac12}\in(\mathcal{L}_{d,\infty})_0,\quad A=\sum_{k=1}^d\gamma_k\otimes A_k.$$
This, together with the fact that $\dbar f\in \mathcal L_{d,\infty}$ (see Proposition \ref{upperbound}) and the algebraic identity
\begin{equation}
	\label{eq:algebraic}
	|X|^2-|Y|^2= - |X-Y|^2 + X^\ast (X-Y)  +(X-Y)^{\ast} X \,,
\end{equation}
implies that that
$$|\dbar f|^2 - |A(1+\mathcal D^2)^{-\frac12}|^2\in(\mathcal{L}_{\frac{d}{2},\infty})_0.$$
Clearly,
$$|A(1+\mathcal D^2)^{-\frac12}|^2=\sum_{k,l=1}^d\gamma_k\gamma_l\otimes (1-\Delta)^{-\frac12}A_k^{\ast}A_l(1-\Delta)^{-\frac12}.$$
Using the anticommutation relations of the gamma matrices, we write
\begin{align*}
&~\quad |A(1+\mathcal D^2)^{-\frac12}|^2-1\otimes (1-\Delta)^{-\frac12}(\sum_{k=1}^d|A_k|^2)(1-\Delta)^{-\frac12}\\
&=\sum_{1\leq k<l\leq d}\gamma_k\gamma_l\otimes (1-\Delta)^{-\frac12}(A_k^{\ast}A_l-A_l^{\ast}A_k)(1-\Delta)^{-\frac12}\\
&=-2i\sum_{1\leq k<l\leq d}\gamma_k\gamma_l\otimes (1-\Delta)^{-\frac12} (\Im A_k)\cdot A_l(1-\Delta)^{-\frac12}\\
&\quad +2i\sum_{1\leq k<l\leq d}\gamma_k\gamma_l\otimes (1-\Delta)^{-\frac12} (\Im A_l)\cdot A_k(1-\Delta)^{-\frac12}\\
&\quad +\sum_{1\leq k<l\leq d}\gamma_k\gamma_l\otimes (1-\Delta)^{-\frac12}[A_k,A_l](1-\Delta)^{-\frac12}.
\end{align*}
The assertion now follows from Lemma \ref{ak properties lemma}, Cwikel's estimate and H\"older's inequality.
\end{proof}

\begin{lem}\label{lemma using main thm} For $f\in C^{\infty}_c(\mathbb{R}^d),$ $\chi\in (C_c+\mathbb{C})(\mathbb{R}^d),$ set
$$T=(\sum_{k=1}^d|A_k|^2)^{\frac12}M_{\chi}.$$
Then
$$\lim_{t\to\infty}t^{\frac1d}\mu(t,T(1-\Delta)^{-\frac12})= \kappa_d' \left( \int_{\mathbb R^{d}} |\chi|^d|\nabla f|^d\,dx \right)^\frac1d$$
where
$$
\kappa_d' = (2\pi)^{-1} \left( d^{-1} \int_{\mathbb S^{d-1}} (1-s_d^2)^{\frac d2} d\mathbf s \right)^\frac1d.
$$	
\end{lem}

\begin{proof} 	We will derive the claimed asymptotics by applying Theorem \ref{main thm} to the operator $T.$ Note that $T\in\Pi$ (since $\Pi$ is a $C^\ast$ algebra, or more explicitly, by approximating the square root in the definition of $T$ by polynomials) and that $T$ is compactly supported from the right (indeed, in the decomposition $L_2(\mathbb R^d)=L_2({\rm supp} \nabla f)\oplus L_2(\mathbb R^d\setminus{\rm supp}|\nabla f|)$ the operator $\sum |A_k|^2$ acts nontrivially only in the first term, and so does its square root). Applying Theorem \ref{main thm}, we deduce that
$$\lim_{t\to\infty} t^\frac1d \mu(t,T(1-\Delta)^{-\frac12})  = d^{-\frac1d} (2\pi)^{-1} \| {\rm sym}(T)\|_d.$$
	
It remains to compute the right side. Since ${\rm sym}:\Pi\to\mathcal{A}_1\otimes_{{\rm min}}\mathcal{A}_2$ is a $\ast$-homomorphism, it follows that
$$|{\rm sym}(T)|=(|\chi|\otimes 1)\cdot(\sum_{k=1}^d |{\rm sym} (A_k) |^2)^{\frac12}.$$
For ${\bf t}\in\mathbb{R}^d$ and ${\bf s}\in\mathbb{S}^{d-1},$ we have
$${\rm sym}(A_k)({\bf t},{\bf s})=(\partial_kf)({\bf t}) -s_k\cdot\langle {\bf s},(\nabla f)({\bf t})\rangle.$$
Thus,
\begin{align*}
|{\rm sym}\,(T)({\bf t},{\bf s})|&= |\chi({\bf t})| \left| (\nabla f)({\bf t}) - {\bf s}\cdot\langle{\bf s},(\nabla f)({\bf t})\rangle\right|\\
&=|\chi({\bf t})|\cdot\big(|(\nabla f)({\bf t})|^2-|\langle{\bf s},(\nabla f)({\bf t})\rangle|^2\big)^{\frac12}
\end{align*}
and
\begin{align*}
\| {\rm sym}\,(T) \|_d^d & = \int_{\mathbb{R}^d\times\mathbb{S}^{d-1}}|\chi({\bf t})|^d \big(|(\nabla f)({\bf t})|^2-|\langle{\bf s},(\nabla f)({\bf t})\rangle|^2\big)^{\frac{d}{2}}d{\bf t}d{\bf s} \\
&=\int_{\mathbb{R}^d}|\chi({\bf t})|^d|(\nabla f)({\bf t})|^d\Big(\int_{\mathbb{S}^{d-1}}\big(|e({\bf t})|^2-|\langle{\bf s},e({\bf t})\rangle|^2\big)^{\frac{d}{2}}d{\bf s}\Big) d{\bf t},
\end{align*}
where
$$e({\bf t})=\frac{(\nabla f)({\bf t})}{|(\nabla f)({\bf t})|},\quad {\bf t}\in\mathbb{R}^d.$$
By rotation invariance,
\begin{align*}
\int_{\mathbb{S}^{d-1}}\big(|e({\bf t})|^2-|\langle{\bf s},e({\bf t})\rangle|^2)^{\frac{d}{2}}d{\bf s} & =\int_{\mathbb{S}^{d-1}}\big(1-|\langle{\bf s},e({\bf t})\rangle|^2)^{\frac{d}{2}}d{\bf s} \\
& =\int_{\mathbb{S}^{d-1}}\big(1-|\langle{\bf s},e_d\rangle|^2)^{\frac{d}{2}}d{\bf s} \,,
\end{align*}
which concludes the proof.
\end{proof}

\begin{rem}\label{const}
	On can express the constant $\kappa_d'$ in terms of gamma functions. Indeed, changing coordinates ${\bf s} = ((\sin\theta){\bf s'},\cos\theta)$ with ${\bf s'}\in\mathbb{S}^{d-2}$ and using $d{\bf s} = d{\bf s'} (\sin\theta)^{d-2}d\theta$, one finds	
\begin{align*}
	\int_{\mathbb{S}^{d-1}}\big(1- s_d^2)^{\frac{d}{2}}d{\bf s}
	= {\rm Vol}(\mathbb{S}^{d-2}) \int_0^\pi \sin^{2d-2}\theta \,d\theta.
\end{align*}
Using
$$
{\rm Vol}(\mathbb{S}^{d-2}) = \frac{2\pi^\frac{d-1}{2}}{\Gamma(\frac{d-1}2)}
$$
and the beta function identity
$$
\int_0^\pi \sin^{2d-2}\theta \,d\theta = B(\tfrac12,d-\tfrac12) = \frac{\sqrt\pi\,\Gamma(d-\frac12)}{(d-1)!} \,,
$$
we obtain
$$
\kappa_d' = (2\pi)^{-1} \left( \frac{2\pi^\frac{d}{2}}{\Gamma(\frac{d-1}2)}\, \frac{\Gamma(d-\frac12)}{d!} \right)^\frac 1d.
$$
Similarly,
\begin{equation}
	\label{eq:const}
	\kappa_d = N^\frac1d \kappa_d' = (2\pi)^{-1} \left(N\, \frac{2\pi^\frac{d}{2}}{\Gamma(\frac{d-1}2)}\, \frac{\Gamma(d-\frac12)}{d!} \right)^\frac 1d.
\end{equation}
\end{rem}

We are now in position to compute the asymptotics of $\dbar f$. We recall that the constant $\kappa_d$ is defined in Theorem \ref{main cor}; see also \eqref{eq:const}.

\begin{prop}\label{asymptotics}
Let $d\geq 2$, let $f\in C^\infty_c(\mathbb R^d)$ be real-valued. For every $\chi\in \mathbb C + C_c(\mathbb{R}^d),$ we have
$$\lim_{t\to\infty} t^\frac1d \mu(t,\dbar f\cdot (1\otimes M_{\chi})) = \kappa_d \left( \int_{\mathbb R^{d}} |\chi|^d|\nabla f|^d\,dx \right)^\frac1d$$
and
$$
\lim_{t\to\infty} t^\frac1d \mu(t,(1\otimes M_{\bar{\chi}})\cdot\dbar f\cdot (1\otimes M_{\chi})) = \kappa_d \left( \int_{\mathbb R^{d}} |\chi|^{2d}|\nabla f|^d\,dx \right)^\frac1d.
$$
\end{prop}
\begin{proof}
	Set
$$
S=(\sum_{k=1}^d|A_k|^2)^{\frac12}(1-\Delta)^{-\frac12}M_{\chi}
$$
and let $T$ be as in Lemma \ref{lemma using main thm}. By Theorem \ref{commutator theorem}, we have
$$S-T(1-\Delta)^{-\frac12}=(\sum_{k=1}^d|A_k|^2)^{\frac12}\cdot [(1-\Delta)^{-\frac12},M_{\chi}]\in(\mathcal{L}_{d,\infty})_0.$$
By Lemma \ref{lemma using main thm} and Lemma \ref{bs sep lemma}, one has
$$\lim_{t\to\infty}t^{\frac1d}\mu(t,S)= \kappa_d' \left( \int_{\mathbb R^{d}} |\chi|^d|\nabla f|^d\,dx \right)^\frac1d.$$
In other words,
\begin{equation}\label{asymptotics eq0}
\lim_{t\to\infty}t^{\frac2d}\mu(t,|S|^2)=(\kappa_d')^2 \left( \int_{\mathbb R^{d}} |\chi|^d|\nabla f|^d\,dx \right)^\frac2d.
\end{equation}	
	
By Lemma \ref{dbarf abs compute lemma}, we have
$$|\dbar f\cdot (1\otimes M_{\chi})|^2-1\otimes |S|^2\in(\mathcal{L}_{\frac{d}{2},\infty})_0.$$
It follows now from \eqref{asymptotics eq0} and Lemma \ref{bs sep lemma} that one has
$$
\lim_{t\to\infty}t^{\frac2d}\mu(t,|\dbar f\cdot (1\otimes M_{\chi})|^2)= N^{\frac2d} (\kappa_d')^2 \left( \int_{\mathbb R^{d}} |\chi|^d|\nabla f|^d\,dx \right)^\frac2d.
$$
Since $N^\frac1d \kappa_d' = \kappa_d$, we obtain the first assertion.

To prove the second one, we use again the fact \cite[Theorem 6.3.1]{LMSZ-book} that
$$
B = \dbar f - A(1+\mathcal{D}^2)^{-\frac12}+B \in (\mathcal{L}_{d,\infty})_0.
$$
Thus,
\begin{align*}&~\quad 
(1\otimes M_{\bar{\chi}})\cdot\dbar f\cdot (1\otimes M_{\chi})-\dbar f\cdot (1\otimes M_{|\chi|^2})\\
&=[(1\otimes M_{\bar{\chi}}),\dbar f]\cdot (1\otimes M_{\chi})\\
&=[(1\otimes M_{\bar{\chi}}),B]\cdot (1\otimes M_{\chi})+\sum_{k=1}^d(\gamma_k\otimes [M_{\bar{\chi}},A_k])\cdot (1\otimes M_{\chi})\\
&=[(1\otimes M_{\bar{\chi}}),B]\cdot (1\otimes M_{\chi})+\sum_{k,j=1}^d(\gamma_k\otimes [M_{\bar{\chi}},\frac{D_kD_j}{\Delta}])\cdot (1\otimes M_{\partial_jf\cdot \chi}).\end{align*}
The first summand on the right hand side belongs to $(\mathcal{L}_{d,\infty})_0$ since so does $B.$ The second summand on the right hand side belongs to $(\mathcal{L}_{d,\infty})_0$ by Theorem \ref{commutator theorem}. Hence,
$$(1\otimes M_{\bar{\chi}})\cdot\dbar f\cdot (1\otimes M_{\chi})-\dbar f\cdot (1\otimes M_{|\chi|^2})\in(\mathcal{L}_{d,\infty})_0.$$
The second assertion now follows from the first one.
\end{proof}

\begin{cor}\label{locbound}
Let $d\geq 2$, let $f\in C^\infty_c(\mathbb{R}^d)$ and $\chi\in C_c(\mathbb{R}^d)$. Then we have
$$\|(1\otimes M_{\bar{\chi}})\cdot\dbar f\cdot(1\otimes M_{\chi})\|_{d,\infty} \geq c_d \left( \int_{\mathbb R^{d}} |\chi|^{2d}|\nabla f|^d\,dx\right)^\frac1d.$$
\end{cor}

\begin{proof} If $f$ is real-valued, the claimed bound is an immediate consequence of Proposition \ref{asymptotics}. For general, complex-valued $f$ we have
$$\Re((1\otimes M_{\bar{\chi}})\cdot\dbar f\cdot(1\otimes M_{\chi}))=(1\otimes M_{\bar{\chi}})\cdot\dbar \Re f\cdot(1\otimes M_{\chi}),$$	
$$\Im((1\otimes M_{\bar{\chi}})\cdot\dbar f\cdot(1\otimes M_{\chi}))=(1\otimes M_{\bar{\chi}})\cdot\dbar \Im f\cdot(1\otimes M_{\chi}).$$	
We now deduce from Lemma \ref{realpart} that
\begin{align*}
& ~\quad \|(1\otimes M_{\bar{\chi}})\cdot\dbar \Re f\cdot(1\otimes M_{\chi})\|_{d,\infty}  + \|(1\otimes M_{\bar{\chi}})\cdot\dbar \Im f\cdot(1\otimes M_{\chi})\|_{d,\infty}  \\
	&  \leq C_d \|(1\otimes M_{\bar{\chi}})\cdot\dbar f\cdot(1\otimes M_{\chi})\|_{d,\infty}  \,.
	\end{align*}
	Meanwhile, by the triangle inequality in $L_d$, we have
	$$
	\left( \int_{\mathbb R^{d}} |\chi|^{2d}|\nabla f|^d\,dx\right)^\frac1d
	\leq \left( \int_{\mathbb R^{d}} |\chi|^{2d}|\nabla \Re f|^d\,dx\right)^\frac1d + \left( \int_{\mathbb R^{d}} |\chi|^{2d}|\nabla \Im f|^d\,dx\right)^\frac1d.
	$$
	Thus, the bound in the complex case follows from that in the real case.
\end{proof}

We are now in a position to prove the second part of Theorem \ref{mainiff}.

\begin{prop}\label{lowerbound}
Let $d\geq 2$ and $f\in BMO(\mathbb R^d)$ with $\dbar f\in\mathcal L_{d,\infty}$. Then $f\in\dot W^1_d(\mathbb R^d)$ with
$$\|\dbar f\|_{d,\infty} \geq c_d\|\nabla f\|_d.$$
\end{prop}

\begin{proof}
	Let $0\leq\Phi\in C^\infty_c(\mathbb{R}^d)$ with $\|\Phi\|_1=1$ and set $\Phi_t({\bf t}):=t^{-d} \Phi(t^{-1}{\bf t}).$ We have $\Phi_t\ast f \in C^\infty(\mathbb{R}^d).$ Let $\chi\in C^\infty_c(\mathbb{R}^d)$ with $|\chi|\leq 1.$ Given $\chi$ we choose a $\tilde\chi\in C^\infty_c(\mathbb{R}^d)$ with $\tilde\chi \chi =\chi.$ We have $\tilde{\chi}\cdot(\Phi_t\ast f)\in C^{\infty}_c(\mathbb{R}^d)$ and
\begin{equation}\label{compactification eq}
(1\otimes M_{\bar{\chi}})\cdot \dbar(\Phi_t\ast f)\cdot (1\otimes M_{\chi}) =(1\otimes M_{\bar{\chi}})\cdot \dbar(\tilde{\chi}\cdot(\Phi_t\ast f))\cdot (1\otimes M_{\chi}).
\end{equation}
	
By a majorization argument (see \cite[Lemma 18]{LMSZ}), we have $\dbar (\Phi_t\ast f) \in \mathcal{L}_{d,\infty}$ and
$$\|\dbar(\Phi_t\ast f)\|_{d,\infty} \leq C_d \|\dbar f\|_{d,\infty} \,.$$
Using \eqref{compactification eq}, we obtain
\begin{equation}\label{lowerbound eq0}
\|(1\otimes M_{\bar{\chi}})\cdot \dbar(\tilde{\chi}\cdot(\Phi_t\ast f))\cdot (1\otimes M_{\chi})\|\leq C_d\|\dbar f\|_{d,\infty} \,.
\end{equation}
Applying Corollary \ref{locbound} to $\tilde{\chi}\cdot(\Phi_t\ast f)\in C^{\infty}_c(\mathbb{R}^d),$ we obtain
\begin{equation}\label{lowerbound eq1}
\|(1\otimes M_{\bar{\chi}})\cdot\dbar(\tilde{\chi}\cdot(\Phi_t\ast f))\cdot(1\otimes M_{\chi})\|_{d,\infty} \geq c_d\||\chi|^2\cdot\nabla(\tilde{\chi}\cdot(\Phi_t\ast f))\|_d \,.
\end{equation}
Combining \eqref{lowerbound eq0} and \eqref{lowerbound eq1} and taking into account that $\nabla\tilde{\chi}=0$ near the support of $\chi,$ we obtain
$$C_d\|\dbar f\|_{d,\infty}\geq c_d\||\chi|^2\cdot\nabla(\Phi_t\ast f)\|_d \,.$$
Taking the supremum over all $\chi\in C^\infty_c(\mathbb{R}^d)$ with $|\chi|\leq 1,$ we deduce that $\nabla(\Phi_t\ast f)\in L_d(\mathbb{R}^d)$ with
$$C_d\|\dbar f\|_{d,\infty}\geq c_d\|\nabla(\Phi_t\ast f)\|_d \,.$$
Since $\Phi_t\ast f\to f$ in $L_{1,{\rm loc}}(\mathbb{R}^d),$ this implies by a standard argument that $f$ is weakly differentiable with $\nabla f\in L_d(\mathbb R^d)$ and
$$C_d\|\dbar f\|_{d,\infty}\geq c_d\|\nabla f\|_d \,,$$
as claimed.
\end{proof}

It remains to prove the validity of the asymptotics under the minimal regularity assumption.

\begin{proof}[Proof of Theorem \ref{main cor}] 
	By \cite[Theorem 11.43]{Leoni-book}, there is a sequence $(f_n)_{n\geq0}\subset C_c^\infty(\mathbb{R}^d)$ with $\nabla f_n\to \nabla f$ in $L_d$. By Proposition \ref{upperbound}, we have $\dbar f_n\to \dbar f$ in $\mathcal{L}_{d,\infty}$ as $n\to\infty.$ By Proposition \ref{asymptotics} (with $\chi=1$), we have
$$\lim_{t\to\infty} t^\frac1d \mu(t,\dbar f_n) = \kappa_d \|\nabla f_n\|_{d},\quad n\geq0.$$	
The assertion now follows from Lemma \ref{bs limit lemma}.	
\end{proof}

\begin{proof}[Proof of Corollary \ref{corconst}] Let $f\in BMO(\mathbb{R}^d)$ be such that
$$\lim_{t\to\infty}t^{\frac1d}\mu(t,\dbar f)=0.$$
In particular, we have $\dbar f\in \mathcal{L}_{d,\infty}.$ By Theorem \ref{mainiff}, we have $f\in\dot W^1_d(\mathbb R^d).$

If $f$ is real-valued, we deduce then from Theorem \ref{main cor} that $\|\nabla f\|_d=0$, that is, $f$ is constant. For complex-valued $f$, we apply Lemma \ref{realpart} and deduce that
	$$
	\limsup_{t\to\infty} t^\frac1d \mu(t,\dbar\Re f) \,,\ \limsup_{t\to\infty} t^\frac1d \mu(t,\dbar\Im f) \leq 2^{\frac1d}\limsup_{t\to\infty} t^\frac1d \mu(t,\dbar f) = 0 \,,
	$$
	so the assertion follows from that in the real-valued case.
\end{proof}


\end{document}